\documentclass[moor]{informs1plain}

\usepackage{placeins}
\usepackage[margin=1in]{geometry}
\usepackage{amsmath}
\usepackage{amssymb}
\usepackage[english]{babel}
\usepackage{color}
\usepackage{empheq}
\usepackage{enumerate}
\usepackage{graphicx}
\usepackage[utf8]{inputenc}
\usepackage{latexsym}
\usepackage{stmaryrd}
\usepackage{subcaption}
\usepackage{txfonts}
\usepackage{tikz}
\tikzstyle{everypicture}+=[remember picture]
\usetikzlibrary{decorations}
\usetikzlibrary{shapes}
\usetikzlibrary{shapes.misc}
\usepackage{pgfplots}

\def\TheoremsNumberedBySection{
\theoremstyle{TH}
\newtheorem{theorem}{Theorem}[section]
\newtheorem{corollary}[theorem]{Corollary}
\newtheorem{lemma}[theorem]{Lemma}
\newtheorem{proposition}[theorem]{Proposition}

\theoremstyle{EX}
\newtheorem{definition}[theorem]{Definition}
\newtheorem{remark}[theorem]{Remark}
}

\TheoremsNumberedBySection
\EquationsNumberedThrough

\usepackage{natbib}
 \NatBibNumeric
 \def\bibfont{\small}%
 \bibpunct[, ]{[}{]}{,}{n}{}{,}%

\renewenvironment{proof}[1][Proof.]{\Trivlist\item[\hspace*{1em}\hskip\labelsep{\it #1\enskip }]\ignorespaces}{\hfill $\blacksquare$ \endTrivlist\addvspace{0pt}}

\newcommand{\R}{\mathbb{R}}
\newcommand{\Q}{\mathbb{Q}}
\newcommand{\Z}{\mathbb{Z}}

\newcommand{\I}{\mathcal{I}}

\def\ve#1{\mathchoice{\mbox{\boldmath$\displaystyle\bf#1$}}
{\mbox{\boldmath$\textstyle\bf#1$}}
{\mbox{\boldmath$\scriptstyle\bf#1$}}
{\mbox{\boldmath$\scriptscriptstyle\bf#1$}}}

\newcommand{\x}{{\ve x}}

\newcommand{\z}{{\ve z}}
\renewcommand{\t}{{\ve t}}
\renewcommand{\b}{{\ve b}}
\renewcommand{\v}{{\ve v}}

\renewcommand{\r}{{\ve r}}

\newcommand{\q}{{\ve q}}

\newcommand{\yy}{{y_*}}
\newcommand{\yprime}{{y'_*}}
\newcommand{\yprimeprime}{{y''_*}}
\newcommand{\xx}{{x}}

\DeclareMathOperator    \conv           {conv}
\DeclareMathOperator    \cone           {cone}

\DeclareMathOperator    \intr                   {int}

\DeclareMathOperator    \rec                    {rec}

\DeclareMathOperator    \vol                    {vol}

\newcommand{\tikzvline}[2][]{
\draw [densely dashed,#1] ({rel axis cs:0,0} -| {axis cs:#2,0}) -- ({rel axis cs:0,1} -| {axis cs:#2,0});
}

\begin{document}
\TITLE{Minimizing Cubic and Homogeneous Polynomials over Integers in the Plane}
\RUNTITLE{Minimizing Cubic and Homogeneous Polynomials over Integers in the Plane}
\ARTICLEAUTHORS{%
\AUTHOR{Alberto Del Pia}
\AFF{Department of Industrial and Systems Engineering \& \\ Wisconsin Institutes for Discovery, University of Wisconsin-Madison, United States}
\AUTHOR{Robert Hildebrand}
\AFF{Institute for Operations Research, ETH Zürich, Switzerland}
\AUTHOR{Robert Weismantel}
\AFF{Institute for Operations Research, ETH Zürich, Switzerland}
\AUTHOR{Kevin Zemmer}
\AFF{Institute for Operations Research, ETH Zürich, Switzerland}
}
\RUNAUTHOR{Del Pia, Hildebrand, Weismantel, Zemmer}
\ABSTRACT{%
We complete the complexity classification by degree of minimizing a polynomial over the integer points in a polyhedron in $\R^2$.  Previous work shows that optimizing a quadratic polynomial over the integer points in a polyhedral region in $\R^2$ can be done in polynomial time, while optimizing a quartic polynomial in the same type of region is NP-hard.  We close the gap by showing that this problem can be solved in polynomial time for cubic polynomials.

Furthermore, we show that the problem of minimizing a homogeneous polynomial of any fixed degree over the integer points in a bounded polyhedron in $\R^2$ is solvable in polynomial time. We show that this holds for polynomials that can be translated into homogeneous polynomials, even when the translation vector is unknown.  We demonstrate that such problems in the unbounded case can have smallest optimal solutions of exponential size in the size of the input, thus requiring a compact representation of solutions for a general polynomial time algorithm for the unbounded case.
}
 
\maketitle

\section{Introduction}

We study the problem of minimizing a polynomial with integer coefficients over the integer points in a polyhedron.  When the polynomial is of degree one, this becomes integer linear programming, which Lenstra \cite{lenstra_integer_1983} showed to be solvable in polynomial time in fixed dimension.  In stark contrast, De Loera et~al.\ \cite{loera_integer_2006} showed that even for polynomials of degree four in two variables, this minimization problem is NP-hard.  For a survey on the complexity of mixed integer nonlinear optimization, see also Köppe \cite{koeppe_complexity_2012}.  Recently, Del Pia et~al.\ \cite{delpia_mixed_2014} showed that the decision version of mixed-integer quadratic programming is in {NP}.  Del~Pia and Weismantel \cite{delpia_integer_2013} showed that for polynomials in two variables of degree two, the problem is solvable in polynomial time.

Consider the problem 
\begin{equation}
\label{eq:main}
\min \{ f(\x) : \x \in P \cap \Z^n\},
\end{equation}
where $P = \{\x \in \R^n : A \x \leq \b\}$ is a rational polyhedron with $A \in \Z^{m \times n}$, $\b \in \Z^m$, and $m, n \in \Z_{\geq 0}$.
Let $d\in \Z_{\geq 0}$ bound the maximum degree of the polynomial function $f$ and let $M$ be the sum of the absolute values of the coefficients of $f$.
We use the words size and binary encoding length synonymously.
The size of $P$ is the sum of the sizes of $A$ and $\b$.
We say that Problem~\eqref{eq:main} can be solved in polynomial time if in time bounded by a polynomial in the size of $A,\b$ and $M$ we can either determine that the problem is infeasible, find a feasible minimizer, or show that the problem is unbounded by exhibiting a feasible point $\bar \x$ and an integer ray $\bar \r\in \rec(P)$ such that $f(\bar \x + \lambda \bar \r) \to -\infty$ as $\lambda \to \infty$.  We almost always assume the degree $d$ and the dimension $n$ are fixed in our complexity results. Moreover, in Sections~\ref{sec:cubic} and \ref{sec:homogeneous} we assume that $P$ is bounded.  Note that if $P$ is bounded, then there exists an integer $R \ge 1$ of polynomial size in the size of $P$ such that $P \subseteq B:= [-R,R]^2$ (see, for instance, \cite{schrijver_theory_1986}).

Previous work has shown that Problem~\eqref{eq:main} is solvable in polynomial time if it is 1-dimensional or the polynomial is quadratic, whereas for $n = 2$, $d = 4$ the problem is NP-hard, even when $P$ is bounded.

\begin{theorem}[\cite{delpia_integer_2013}, 1-dimensional polynomials and quadratics]
\label{thm:dimension1}
 Problem~\eqref{eq:main} is solvable in polynomial time when $n = 1$ with $d$ fixed, and when $n = 2$ with $d \leq 2$.
\end{theorem}

\begin{lemma}[\cite{loera_integer_2006}]
\label{lem:deg4-np-hard} 
Problem~\eqref{eq:main} is NP-hard when $f$ is a polynomial of degree $d = 4$ with integer coefficients and $n = 2$, even when $P$ is bounded.
\end{lemma}

Using the same reduction as Lemma~\ref{lem:deg4-np-hard}, it is possible to show that Problem~\eqref{eq:main} is NP-hard even when $n = d = 2$, $P$ is a bounded, rational polyhedron, and we add a single quadratic inequality constraint (see \cite{manders_np_1976}).

We improve Theorem~\ref{thm:dimension1} to the case $n=2$ and $d = 3$.
\begin{theorem}[cubic]
\label{thm:min_cubic_unbounded}
Problem~\eqref{eq:main} is solvable in polynomial time for $n = 2$ and $d = 3$.
\end{theorem}
We prove this theorem under the assumption that $P$ is bounded in Section~\ref{sec:cubic}, and without this additional assumption in Section~\ref{sec:cubic-unbounded}.
Thus, we complete the complexity classification by degree $d$ for Problem~\eqref{eq:main} when $n=2$.
It is an open question whether Problem~\eqref{eq:main} can be solved in polynomial time for $n \ge 3$ and $d \in \{2, 3\}$.

Problem~\eqref{eq:main} remains difficult even when the polynomials are restricted to be homogeneous and the degree is fixed.  
The polynomial $h$ is \emph{homogeneous} of degree $d$ if 
\begin{equation}
h(\x) = \sum_{\v \in \Z^n_+, \|\v\|_1 = d} c_\v \x^\v ,
\end{equation}
where $c_\v \in \R$, $\|\cdot \|_1$ denotes the 1-norm, and $\x^\v = \prod_{i=1}^n x_i^{v_i}$.
The case of general polynomials in $n$ variables reduces to the case of homogeneous polynomials in $n+1$ variables by homogenizing $f(\x)$ using an additional variable $x_{n+1}$ and adding the constraint $x_{n+1} = 1$ to $P$.
Thus, complexity results for general polynomials provide partial complexity results for homogeneous polynomials.

\begin{proposition}
\label{prop:NP-hard-hom}
Problem~\eqref{eq:main} is NP-hard when $f$ is a homogeneous polynomial of degree $d$ with integer coefficients, $n \geq 3$ and $d \ge 4$ are fixed, even when $P$ is bounded.
\end{proposition}

We next show that we cannot expect tractable size solutions to unbounded homogeneous minimization problems in dimension two. 

\begin{proposition}
\label{prop:pell}
There exists an infinite family of instances of Problem~\eqref{eq:main} with $f$ homogeneous, $d = 4$, $n=2$ such that the minimal size solution to Problem~\eqref{eq:main} has exponential size in the input size.
\end{proposition}
\begin{proof}
  Consider the minimization problem
\begin{equation}
\label{eq:pell-opt}
\min \{\big(x^2 - N y^2\big)^2 : (x,y) \in P \cap \Z^2\},
\end{equation}
where $P=\{(x,y) \in \R^2 : x \ge 1, \, y \ge 1\}$ is an unbounded rational polyhedron and $N$ is a nonsquare integer.  The objective function is a homogeneous bivariate polynomial of degree four.  
Since $(0,0) \notin P$, $(x^2 - N y^2)^2$ is nonnegative, and since $N$ is nonsquare, the optimum of Problem~\eqref{eq:pell-opt} is strictly greater than zero.  Note that $(x^2 - N y^2)^2 =1$ if and only if $(x,y)$ is a solution to either the \emph{Pell equation}, $x^2 - N y^2 = 1$, or the \emph{Negative Pell equation}, $x^2 - N y^2 = -1$.  The Pell equation has an infinite number of positive integer solutions (see, for instance, \cite{Walker1967}) and therefore, we infer that the optimum of Problem~\eqref{eq:pell-opt} equals $1$.  

Lagarias \cite[Appendix A]{lagarias_computational_1980} shows that the Negative Pell equation with $N = 5^{2k + 1}$ has solutions for every $k \ge 1$ and that the solution $(x^*, y^*)$ to this equation with minimal size satisfies
\begin{equation*}
x^*+y^* \sqrt 5 = (2 + \sqrt 5)^{5^k}.
\end{equation*}
The method is based on principles due to Dirichlet \cite{dirichlet_propriete_1856}.
This implies that while the input is of size $O(k)$, any solution to the Negative Pell equation expressed in binary form for these $N$ has size $\Omega(5^k)$.

Theorem 6.10 of \cite{Cameron} (see also \cite{Walker1967}) shows that if the Negative Pell equation has a solution, then the minimal size solution to $x^2 - N y^2 = \pm 1$ is in fact the minimal size solution to the Negative Pell equation.  
Therefore, any solution to Problem~\eqref{eq:pell-opt} with $N= 5^{2k + 1}$ has an exponential size in the size of the input.  
 Since Problem~\eqref{eq:pell-opt} is has an objective function that is homogeneous of degree four and has linear constraints, this finishes the proof.
\end{proof}

For bounded polyhedra $P$, we will show that Problem~\eqref{eq:main} is solvable in polynomial time for any fixed degree in two variables when the objective function is a polynomial that is a coordinate translation of a homogeneous polynomial.
A polynomial $f(\x)\colon \R^n \to \R$ is \emph{homogeneous translatable} if there exists $\t \in \R^n$ such that $f(\x+\t) = h(\x)$ for some homogeneous polynomial $h(\x)$.  In our results, we will assume that we are given a homogeneous translatable polynomial $f$ with integer coefficients, but that we are not given the translation $\t$.  Our algorithmic techniques apply to this natural generalization of homogeneous polynomials without even needing to solve for $\t$.  Even so, for $n=2$, we show in the Appendix (Proposition~\ref{prop:ht}) that in polynomial time we can check if $f$ is homogenous translatable and produce a rational $\t$ if it is.

\begin{theorem}[homogeneous translatable, bounded]
\label{thm:min_homogeneous_bounded}
Problem~\eqref{eq:main} is solvable in polynomial time for $n = 2$ and any fixed degree $d$, provided that $f$ is homogeneous translatable and $P$ is bounded.
\end{theorem}

This theorem highlights the fact that the complexity of bounded polynomial optimization with two integer variables is not necessarily related to the degree of the polynomials, but instead to the difficulty in handling the lower order terms.

Despite the possibly large size of solutions to minimizing homogeneous polynomials of degree four (see Proposition~\ref{prop:pell}), Theorem~\ref{thm:min_cubic_unbounded} shows that we can solve the unbounded case for degree three.

The details of our proofs strongly rely on the properties of cubic and homogeneous polynomials.
When $f \colon\R^2 \to \R$ is a quadratic polynomial, \cite{delpia_integer_2013} proves Theorem~\ref{thm:dimension1} using the fact that $P$ can be divided into regions where $f$ is quasiconvex and quasiconcave.  We use a similar approach for homogeneous polynomials and determine quasiconvexity and quasiconcavity by analyzing the bordered Hessian. We show that the bordered Hessian can be well understood for homogeneous polynomials.  For general polynomials, these regions cannot in general be described by hyperplanes and are much more complicated to handle, even for the cubic case.

In Section~\ref{sec:preliminaries}, we present the tools for the main technique of the paper.  This technique is based on an operator that determines integer feasibility on sets $P \cap C$ and $P \setminus C$, where $P$ is a polyhedron, $C$ is a convex set, and the dimension is fixed.  It relies on two important previous results, namely that in fixed dimension the feasibility problem over semialgebraic sets can be solved in polynomial time \cite{khackiyan_integer_2000}, and the vertices of the integer hull of a polyhedron can be computed in polynomial time \cite{cook_integer_1992,hartmann-1989-thesis}.  We employ this operator to solve the feasibility problem by dividing the domain into regions where this operator can be applied.

In Section~\ref{sec:numerical}, we give some results related to numerically approximating roots of univariate polynomials, which we use throughout this paper.  We show how we can find inflection points of a particular function derived from the quadratic equation using these numerical approximations, which will play a key role in Section~\ref{sec:cubic}.

In Section~\ref{sec:cubic}, we prove Theorem~\ref{thm:min_cubic_unbounded} under the assumption that $P$ is bounded.  We do this by dividing the feasible domain into regions where either the sublevel or superlevel sets of $f$ can be expressed as a convex semialgebraic set.  With this division in hand, the operator presented in Section~\ref{sec:preliminaries} is then applied.

In Section~\ref{sec:homogeneous}, we derive a similar division description of the feasible domain for homogeneous polynomials. While for cubic polynomials the division description depends on the individual sublevel sets, there is a single division description that can be used for all sublevel sets of a particular homogeneous function.
We separate the domain into regions where the objective function is quasiconvex or quasiconcave. These regions allow us to use the operator from Section~\ref{sec:preliminaries}, establishing the complexity result of Theorem~\ref{thm:min_homogeneous_bounded}. 

In Section~\ref{sec:cubic-unbounded}, we consider again cubic polynomials, but relax the requirement that $P$ is bounded, and thus prove Theorem~\ref{thm:min_cubic_unbounded}.

\section{Operator on Convex Sets and Polyhedra}
\label{sec:preliminaries}
Our main approach for solving Problem~\eqref{eq:main} for bounded $P$ is to instead solve the feasibility problem.  As is well known, the feasibility problem and the minimization problem are polynomial time equivalent via reduction with the bisection method, given that appropriate bounds on the objective are known.  We summarize this here. Given a function $f : \R^2 \to \R$ and $\omega \in \R$, define $S^f_{* \omega} := \{ (x,y)\in \R^2 : f(x,y) * \omega\}$ for $* \in \{ \leq, \geq, <, >, =\}$.

\begin{lemma}[feasibility to optimization]
\label{lem:feasibility_to_optimization}
Let $f$ be a bivariate polynomial of fixed degree $d$ with integer coefficients and suppose that $P$ is bounded. Then, if for each $\omega \in \Z +  \tfrac{1}{2}$ we can decide in polynomial time whether the set $S_{\le \omega}^f \cap \Z^2$ is empty or not, we can solve Problem~\eqref{eq:main} in polynomial time.
\end{lemma}
\begin{proof}
Since $P\subseteq B = [-R,R]^2$ and  $f$ is a polynomial of degree $d$, it follows that $- M R^d \le f(x, y) \le M R^d$. Thus, the result is a simple application of binary search on values of $\omega$ in $[-MR^d, M R^d] \cap \left( \Z + 1/2 \right)$.
\end{proof}

We consider $\omega \in \Z + \tfrac{1}{2}$ since this implies $S^f_{\leq \lfloor \omega\rfloor} \cap \Z^2 = S^f_{\leq \omega} \cap \Z^2$ because $f$ has integer coefficients.  Furthermore, this implies that $S^f_{=\omega} \cap \Z^2 = \emptyset$, and therefore $S^f_{\leq \omega}\cap \Z^2 = S^f_{< \omega} \cap \Z^2$ and similarly $S^f_{\geq \omega} \cap \Z^2 = S^f_{> \omega} \cap \Z^2$.  This is important for the proof of Theorem~\ref{lem:cubic_y_deg3}.

A \emph{semi-algebraic set} in $\R^n$ is a subset of the form 
$\bigcup_{i=1}^s \bigcap_{j=1}^{r_i} \left\{ \x \in \R^n \,\, | \,\, f_{i,j}(\x)\ *_{i,j}\ 0 \right\}$
where $f_{i,j} : \R^n \to \R$ is a polynomial in $n$ variables and $*_{i,j}$ is either $<$ or $=$ for $i = 1, \dots, s$ and $j = 1, \dots, r_i$ (cf.\ \cite{bochnak_real_1998}).

\begin{lemma}[polyhedra/convex set operator]
\label{lem:feasibility-test}
Let $P,C \subseteq \R^n$ be such that $P$ is a rational, bounded polyhedron, $C$ is given by a membership oracle, $P \cap C$ is convex, and $n \in \Z_{\ge 1}$ is fixed. In polynomial time in the size of $P$, we can determine a point in the set $(P \setminus C) \cap \Z^n$ or assert that it is empty. Moreover, if $C$ is semi-algebraic and given by polynomial inequalities of degree at most $d \ge 2$ and with integral coefficients of size at most $l$, in polynomial time in $d$, $l$ and the size of $P$, we can determine a point in $P \cap C \cap \Z^n$ or assert that it is empty.
\end{lemma}
\begin{proof}
We can determine whether or not $(P \setminus C) = \emptyset$ by first computing the integer hull $P_I$ of $P$ in polynomial time using \cite{cook_integer_1992,hartmann-1989-thesis}.   Next, we test whether all of its vertices lie in $C$. If they all lie in $C$, then by convexity of $C$ we have that $P \cap C \cap \Z^n \subseteq C$, thus $(P \setminus C) \cap \Z^n$ is empty. Otherwise, since vertices are integral, we have found an integer point in $(P \setminus C) \cap \Z^n$.

Next, since $P \cap C$ is a convex, semialgebraic set, by \cite{khackiyan_integer_2000} we can determine in polynomial time whether $P \cap C\cap \Z^n$ is empty, and if it is not, compute a point contained in it.
\end{proof}

If we can appropriately divide up the feasible domain into regions of the type that Lemma~\ref{lem:feasibility-test} applies to, then we are able to solve Problem~\eqref{eq:main} in polynomial time. We formalize this in the remainder of this section.

\begin{definition}
\label{def:division-description}
Given a sublevel set $S^f_{\leq \omega}$ and a box $B= [-R,R]^2$, a \emph{division description} of the sublevel set on $B$ is a list of rational polyhedra $P_i$, $i = 1, \dots, l_1$, $Q_j$, $j = 1, \dots, \ell_2$, and rational lines $L_k$, $k = 1, \dots, \ell_3$, such that
\begin{enumerate}[(i)]
\item $P_i \cap S^f_{\leq \omega}$ is convex for $i=1, \dots, \ell_1$,
\item $Q_j \setminus S^f_{\leq \omega} = Q_j \cap S^f_{> \omega}$ is convex for $j=1, \dots, \ell_2$, 
\item and 
\begin{equation}
B\cap \Z^2 = \left(\bigcup_{i=1}^{\ell_1} P_i \cup \bigcup_{j=1}^{\ell_2} Q_j \cup \bigcup_{k=1}^{\ell_3} L_k\right) \cap \Z^2 .
\end{equation}
\end{enumerate}
\end{definition}

We will create division descriptions of sublevel sets $S^f_{\leq \omega}$ on a box $B:= [-R,R]^2$ with $P\subseteq B$.

\begin{theorem}
\label{thm:dd-min}
Suppose $P$ is bounded, and for every $\omega \in \Z + \tfrac{1}{2}$ with $\omega \in [-M R^d, M R^d]$, we can determine a division description for $S^f_{\leq \omega}$ on $B$ in polynomial time.  Then we can solve Problem~\eqref{eq:main} in polynomial time. 
\end{theorem}
\begin{proof}
Follows from Lemmas~\ref{lem:feasibility_to_optimization} and \ref{lem:feasibility-test}.
\end{proof}

\section{Numerical Approximations and the Quadratic Formula}
\label{sec:numerical}

For a finite set $\mathcal A = \{\alpha_0 = -R, \alpha_1, \ldots, \alpha_k, \alpha_{k+1} = R\} \subseteq \R$, with $\alpha_i < \alpha_{i+1}$, $i = 0, \dots, k$, we define the set of points $X_{\mathcal A} := \{ \lfloor \alpha_i \rfloor, \lceil \alpha_i \rceil  : i=0, \dots, k+1\}$ and the set of intervals $\I_{\mathcal A} := \{[\lceil \alpha_i \rceil + 1 , \lfloor \alpha_{i+1} \rfloor - 1]: i=0, \dots, k\}$ (some of which may be empty).
Notice that $[-R,R] \cap \Z = (X_{\mathcal A} \cup \bigcup_{I \in \I_{\mathcal A}} I) \cap \Z$.  Therefore, a minimizer $\x^* \in \argmin \left\{ f(\x) : \x \in P \cap \Z^2 \right\}$ where $P \subseteq B:= [-R,R]^2$ is attained either on a set $P \cap (\{x\} \times \R)$ for some $x \in X_{\mathcal A}$ or on a set $P \cap (I \times \R)$ for some $I \in \I_{\mathcal A}$.  Solving the minimization problem on each of those sets separately and taking the minimum of all problems will solve the original minimization problem in $P \cap \Z^2$. We use this several times with $\mathcal A$ being an approximation to the roots, extreme points, or inflection points of some univariate function.

\begin{lemma}[numerical approximations]
\label{lem:numerical}
Let $p$ be a univariate polynomial of degree $d$ with integer coefficients, and suppose that its coefficients are given. Let $M$ be the sum of the absolute values of the coefficients of $p$, and let $\epsilon> 0$ be a rational number.  
\begin{enumerate}[(i)]
\item In polynomial time in $d$ and the size of $M$, we can determine whether or not $p \equiv 0$. \label{part1}
\item  Suppose $p \not\equiv 0$ and $\alpha_1, \dots, \alpha_k$ are the real roots of $p$.  Then, in polynomial time in $d$ and the size of $M$ and $\epsilon$, we can find a list of rational numbers $\tilde \alpha_1, \ldots, \tilde \alpha_k$ of $\epsilon$-approximations of $\alpha_1, \dots, \alpha_k$, that is, $|\alpha_i - \tilde \alpha_i| < \epsilon$ for $i=1,\dots,k$. \label{part2}
\item Suppose $p \not\equiv 0$ and $\alpha_1, \dots, \alpha_k$ are the distinct real roots of $p$ in increasing order.  Then, in polynomial time in $d$ and the size of $M$ and $\epsilon$,  we can determine a list of rational numbers $\tilde \alpha_1^- < \tilde \alpha_1^+ < \dots < \tilde \alpha_k^- < \tilde \alpha_k^+$ such that $\tilde \alpha^-_i <  \alpha_i < \tilde \alpha^+_i$ and $|\tilde \alpha_i^- - \tilde \alpha_i^+| < \epsilon$ for $i = 1, \dots, k$. \label{part3}
\end{enumerate}
\end{lemma}
\begin{proof}
If all coefficients of $p$ are equal to zero, then $p \equiv 0$.  Otherwise $p \not \equiv 0$, proving part \eqref{part1}.
Parts \eqref{part2} and \eqref{part3} follow, for example, from \cite{sagraloff_computing_2013}.
\end{proof}

We use Lemma~\ref{lem:numerical} repeatedly in the following sections.  One way we will use it is in the form of the following remark.
\begin{remark}
\label{rem:numerical}
By choosing $\epsilon$ sufficiently small, for example $\epsilon = 1/4$, we can use Lemma~\ref{lem:numerical} part~\eqref{part2} to determine approximations $\mathcal A = \{\tilde \alpha_1, \dots, \tilde \alpha_k\}$ of the roots of $p$ such that no interval $I \in \I_{\mathcal A}$ contains a root of $p$. Thus, by continuity, $p$ does not change sign on each interval. Moreover, if an interval $I \in \I_{\mathcal A}$ is non-empty, we can determine whether $p$ is positive or negative on $I$ by testing a single point in the interval. 
\end{remark}

The next lemma will be crucial for proving Lemma~\ref{lem:cubic_y_deg2}.

\begin{lemma}
\label{lem:quadratics}
Let $f_0,f_1,f_2 \colon \R \to \R$ be polynomial functions in one variable of fixed degree and suppose that $f_2 \not \equiv 0$. Consider the two functions 
\begin{equation*}
y_\pm := \frac{-f_1 \pm \sqrt{\Delta}}{2f_2}
\end{equation*}
where $\Delta = f_1^2 - 4f_2f_0$.
In polynomial time, we can find a set of rational points $\mathcal A = \{\tilde \alpha_1, \dots, \tilde \alpha_k\}$ such that $y_\pm$ are well defined, continuous and either convex or concave on each $I \in \I_{\mathcal A}$. Moreover, we can determine numbers $c_\pm^I \in \{-1,1\}$ that indicate whether $y_\pm$ is convex or concave on $I \in \I_{\mathcal A}$.

\end{lemma}
\begin{proof}
We start with $\mathcal A = \emptyset$. Since $f_2$ is not identically equal to zero, the number of its zeros is bounded by the degree of $f_2$. By Lemma~\ref{lem:numerical} part~\eqref{part2}, we can approximate its zeros with $\epsilon = 1/8$, which we add to the list $\mathcal A$. We do the same for the zeros of $\Delta$.

We will show the result for $y_+$ only, since the computation is analogous for $y_-$.  
Then 
$$
y'_+ =  \frac{\sqrt{\Delta} \big( f_2' f_1^{} - f_2^{}  f_1' \big)-\Delta   f_2' +\frac{1}{2} f_2^{}  \Delta' }{2 f_2^2\sqrt{\Delta}}
=  \frac{\Delta  \big( f_2' f_1^{} - f_2^{}  f_1' \big)+ (-\Delta   f_2' +\frac{1}{2} f_2^{}  \Delta')\sqrt{\Delta} }{2 f_2^2\Delta} ,
$$

\begin{equation*}
y_+'' =  \frac{-p + \sqrt{\Delta} q}{8 f_2^3 \Delta^{3/2}} ,
\end{equation*}

where 
\begin{align*}
-p &:= \Delta   \left(2 f_2^2 \Delta'' -4 f_2^{}  f_2'  \Delta' \right)+\Delta^2 \left(8 (f_2')^2-4 f_2^{}  f_2'' \right)-f_2^2 (\Delta')^2 , \\
q &:= \Delta \left(4 f_2^{}  f_1^{}  f_2'' +8 f_2^{}  f_2'  f_1' -8 f_1^{}  (f_2')^2-4 f_2^2 f_1'' \right) .
\end{align*}

Therefore, $y'' = 0$ if and only if we have
$
p = q \sqrt{\Delta}.
$
It can be checked that its solutions are exactly the solutions of 
$
p^2 = q^2\Delta
$
and 
$
pq \geq 0.
$
We can determine integer intervals where $pq \geq 0$ by computing $\epsilon$-approximations of the zeros of $pq$ using Lemma~\ref{lem:numerical} part~\eqref{part1} and part~\eqref{part2}.  Note that if $pq \equiv 0$, then there is just one interval, which is $\R$.  Moreover, we can determine whether $p^2 - q^2 \Delta \equiv 0$. If it is, then we add the approximations of the zeros of $pq$ to $\mathcal A$.
Otherwise, we compute $\epsilon$-approximations of the zeros of $p^2 - q^2 \Delta$ and add those to $\mathcal A$.
Finally, we can determine the convexity or concavity of $y_+$ on each non-empty interval $I \in  \I_{\mathcal A}$ by evaluating the sign of $y_+''$ on an point of this interval.
\end{proof}

In the absence of exact computation of irrational roots, we must make up for the error.  In Sections~\ref{sec:cubic} and \ref{sec:homogeneous} we will use our numerical approximations to construct thin boxes containing irrational lines.
\begin{lemma}
\label{lem:line}
Let $K \subseteq \R^2$ be a polytope with $\vol(K) < \tfrac12$.  Then $\dim(\conv(K \cap \Z^2)) \leq 1$ and in polynomial time we can determine a line containing all the integer points in $K$.
\end{lemma}
\begin{proof}
The fact that $\dim(\conv(K \cap \Z^2)) \leq 1$ is a well known result.  See, for example, \cite{oertel_mirror_2013} for a proof.  Using Lenstra's algorithm~\cite{lenstra_integer_1983}, in polynomial time we can either find an integer point $\bar \x \in K \cap \Z^2$, or determine that no such point exists.  If no point exists, then return any line.  Otherwise, let $K_1 = K \cap \{ \x : x_1 \leq \bar x_1 - 1\}$ and $K_2 = K \cap \{ x : x_1 \geq \bar x_1 + 1\}$, and use Lenstra's algorithm twice to detect integer points in the sets $K_1 \cap \Z^2$ and $K_2 \cap \Z^2$.  If an integer point $\tilde \x \in \left( K_1 \cup K_2 \right) \cap \Z^2$ is detected, then return the line given by the affine hull of $\{\bar \x, \tilde \x\}$.  Otherwise return the line $\{ \x : x_1 = \bar x_1\}$.
\end{proof}

\section{Cubic Polynomials}
\label{sec:cubic}

In this section we will prove that Problem~\eqref{eq:main} is solvable in polynomial time for $n=2$ and $d=3$ when $P$ is bounded.  For the rest of this section, let $f(x,y)$ be a bivariate cubic polynomial.
We represent $f(x,y)$ in terms of  $y$ as
$$
f(x,y) = \sum_{i=0}^3 f_i(x) y^i = f_0(x) + f_1(x)y + f_2(x) y^2 + f_3(x) y^3.
$$
Let $\deg_y(f)$ denote the maximum index $i$ such that $f_i$ is not the zero polynomial.
Given a similar representation in terms of $x$, we can similarly define $\deg_x(f)$.  Without loss of generality, we can assume that $\deg_x(f) \geq \deg_y(f)$.  
We will consider each case $\deg_y(f) = 0, 1, 2, 3$ separately.

\begin{theorem}
\label{thm:polynomial_root_bounds}
Let $m < n$ be nonnegative integers, $a_m, \ldots, a_n \in \R$, $a_m \ne 0$, $a_n \ne 0$, and let $\bar x \in \R$ be a nonzero root of the polynomial $f(x) := \sum_{i=m}^n a_i x^i$. Then
\begin{equation}
\min \{ |a_m|/(|a_m| + |a_i|) \, : \, i = m+1, \ldots, n \} < |\bar x| < 1 + \max \{ |a_i/a_n| \, : \, i = m, \ldots, n-1 \} .
\end{equation}
\end{theorem}
\begin{proof}
Follows from Rouché's theorem. See, for example, Theorem~(27,2) in \cite{marden_geometry_1949} for the second inequality. The first inequality can be obtained from the second one by considering the polynomial $g(x) := x^n f(1/x)$.
\end{proof}

\begin{definition}
A bivariate polynomial is called \emph{affinely critical} if the set of critical points, i.e., points where the gradient vanishes, is a finite union of affine spaces---i.e., all of $\R^2$, lines, or points.
\end{definition}

\begin{lemma}
\label{lem:cubic-easy}
All cubic polynomials in two variables are affinely critical. 
\end{lemma}
\begin{proof}
Consider a cubic polynomial $f(x, y)$ in two variables. Since it has degree at most three, both components of its gradient have degree at most two. Thus the gradient vanishes on the intersection of two conic sections (i.e., quadrics in the Euclidean plane). If one of the conic sections is a line, then its intersection with the other conic section is either a line or a finite number of points. Thus suppose that neither of the two conics is a single line. If the two conic sections are distinct, then their intersection consists of at most four distinct points. Therefore suppose that they are not distinct, which happens when $f_x = a f_y$ for some $a \in \R$, where $f_x$, $f_y$ are the derivatives of $f$ with respect to $x$ and $y$ respectively. By equating coefficients, a straightforward calculation shows that  
\begin{equation*}
f(x, y) = c_3 (a x+y)^3 + c_2 (a x+y)^2 + c_1 (a x+y) + c_0,
\end{equation*}
where $c_0, c_1, c_2, c_3$ are a subset of the coefficients of the original polynomial. The gradient of $f$ vanishes if and only if 
\begin{equation}
\label{eq:c-equation}
3 c_3 (a x+y)^2 + 2 c_2 (a x+y) + c_1 = 0 .
\end{equation}
If $c_3 = c_2 = 0$, then this is either the empty set or all of $\R^2$, depending on whether $c_1 \ne 0$ or $c_1 = 0$. If $c_3 = 0$ and $c_2 \ne 0$, then equation~\eqref{eq:c-equation} reduces to the line $a x + y = - \frac{c_1}{2 c_2}$. Finally, if $c_3 \ne 0$, then the gradient vanishes if and only if 
\begin{equation*}
a x + y = - \frac{c_2 \pm \sqrt{c_2^2 - 3 c_1 c_3}}{3 c_3} ,
\end{equation*}
which is either the union of two real lines, or is not satisfied by any real points, depending on whether $c_2^2 - 3 c_1 c_3 \ge 0$ holds or not.
\end{proof}

We now start with showing that Problem~\eqref{eq:main} can be solved in polynomial time if $\deg_y(f) = 0$.

\begin{lemma}
\label{lem:cubic_y_deg0}
Suppose $deg_y(f) = 0$. Then we can solve Problem~\eqref{eq:main} in polynomial time.
\end{lemma}
\begin{proof}
The possible extreme points of the one-dimensional function $f_0$ correspond to the zeros of its first derivative, say $\alpha_1, \ldots, \alpha_k$ where $k \leq 3$.  By Remark~\ref{rem:numerical}, we can determine a list of approximations $\mathcal A = \{\tilde \alpha_1, \dots, \tilde \alpha_k\}$ and define $\I_{\mathcal A}$ and $X_{\mathcal A}$ as in Section~\ref{sec:numerical}. We then solve the problem on each restriction of $P$ to $\{x\} \times \R$ for each $x \in X_{\mathcal A}$ using Theorem~\ref{thm:dimension1}, since this problem is one dimensional.   For each interval $I \in \I_{\mathcal A}$, $f_0(x)$ is either increasing or decreasing in $x$.  Therefore, the optimal solution restricted to the interval is an optimal solution to one of the problems $\min/\max \{ x : (x, y) \in P \cap \Z^2, x \in I \}$.  These problems are just integer linear programs in fixed dimension that are well known to be polynomially solvable (see Scarf~\cite{Scarf1981,Scarf1981a} or Lenstra~\cite{lenstra_integer_1983}).  Since $|\I_{\mathcal A}| \le 3$, the algorithm takes polynomial time.
\end{proof}

For the remaining cases, we solve the feasibility problem instead and rely on Lemma~\ref{lem:feasibility_to_optimization}  to solve the corresponding optimization problem. Moreover, we only need to find a division description for $S_{\leq \omega}^f$, because then we can solve the feasibility problem by Theorem~\ref{thm:dd-min}.

\begin{lemma}
\label{lem:cubic_y_deg1}
Suppose $deg_y(f) = 1$.  For any $\omega \in \Z+\tfrac{1}{2}$, we can find a division description for $S_{\leq \omega}^f$ on $B$ in polynomial time.
\end{lemma}
\begin{proof}
Since $f_1 \not \equiv 0$, apply Lemma~\ref{lem:numerical} part~\eqref{part2} to find approximate roots $\tilde \alpha_1, \dots, \tilde \alpha_k$ of $f_1(x) = 0$ with $k \leq 2$ and an approximation guarantee of $\epsilon = 1/4$.  Hence, for all intervals $I \in \I_{\mathcal A}$, we know that $f_1(x) \neq 0$ $\forall x \in I$.  We now consider solutions $(x,y) \in S_{=\omega}^f$ and see that we can write $y$ as a function of $x$ by rewriting  $f(x, y) = \omega$.  We denote this function by $\yy$ and compute it and its derivatives $\yprime, \yprimeprime$ with respect to $x$. 

\begin{align*}
\yy(x) &= \frac{\omega - f_0(x)}{f_1(\xx)},\\
\yprime(x) &=  \frac{\left(f_0(\xx)-\omega \right) f_1'(\xx)-f_1(\xx) f_0'(\xx)}{f_1(\xx){}^2},\\
\yprimeprime(x) &=  \frac{f_1(\xx) \left(\left(f_0(\xx)-\omega \right) f_1''(\xx)+2 f_0'(\xx) f_1'(\xx)\right)+2 \left(\omega -f_0(\xx)\right) f_1'(\xx){}^2+f_1(\xx){}^2 \left(-f_0''(\xx)\right)}{f_1(\xx){}^3}.
\end{align*}
Let $N(\xx)$ be the numerator of $\yprimeprime$, so $\yprimeprime= N(\xx) / (f_1(\xx)^3)$.  Using Lemma~\ref{lem:numerical} part~\eqref{part1}, we can check whether $N(\xx) \equiv 0$.

\textbf{\underline{Case 1:} $N(\xx) \equiv 0$.}
If $N(\xx) \equiv 0$, then $\yprime$ is constant, and hence $\yy$ is an affine function on each interval $I \in \I_{\tilde \alpha}$. Since $\deg_y(f) = 1$, the sublevel set is either the epigraph or the hypograph of the affine function $y_*$.  Thus, $X_{\mathcal A}$ and $\I_{\mathcal A}$ yield a division description.

\textbf{\underline{Case 2:} $N(x)\not\equiv 0$.}
We use Lemma~\ref{lem:numerical} part~\eqref{part2} to find $1/4$-approximations $\mathcal B = \{\tilde \beta_1, \dots, \tilde \beta_{k'}\}$ of the roots of $N(x) = 0$. The division description is then given by $X_{\mathcal A \cup \mathcal B}$ and $\I_{\mathcal A \cup \mathcal B}$, since the curve has no inflection points on these intervals.
\end{proof}

A crucial tool for the next lemma is the following consequence of Bézout's theorem.

\begin{remark}
\label{rem:three-intersections}
Let $f(x,y)$ be a cubic polynomial and let $L = \{(x,y) : a x + b y + c = 0\}$ be any line with either $a \ne 0$ or $b \ne 0$.  Then either $L$ is contained in the level set $S_{= \omega}^f$, or they intersect at most three times.  When $b \ne 0$ ($a \ne 0$ is analogous), this is because $f(x, - \frac{ax + c}{b})$ is a cubic polynomial in $x$, which is either the zero polynomial, or has at most three zeros. 
\end{remark}

\begin{figure}
\centering
\begin{subfigure}[b]{0.23\textwidth}
\centering
\begin{tikzpicture}[scale=0.4,font = \LARGE,
declare function={ a(\x)  = 8+4*\x; 
                              b(\x) = -32;
                              c(\x) = \x^3 - 4*\x^2 - 16*\x + 32;
                              yplus(\x) = (-1*b(\x) + sqrt(b(\x)*b(\x) - 4*a(\x)*c(\x)))       /(2*a(\x));
                              yminus(\x) = (-1*b(\x) - sqrt(b(\x)*b(\x) - 4*a(\x)*c(\x)))/(2*a(\x)); } ]
                              
\begin{axis}[no markers,domain=-2:5,samples=100,xmin=-1,xmax=5,ymin=-2,ymax=+4]

\path[fill,black!50!white, opacity=0.1]  (axis cs:1,-4) rectangle (axis cs:4,4);

\addplot[black, mark options={solid}] coordinates {(1,1)[$(\ell, \tilde y_\ell)$] (4,0)[$(u, \tilde y_u) $]};
\addplot[blue,thick] {yplus(x)};
\addplot[blue,thick] {yminus(x)};

\node at (axis cs:3,2.6) {$y_+$};
\node at (axis cs:3,-1.3) {$y_-$};
\node[circle,fill,inner sep=2pt] at (axis cs:1,1) {};
\node[label={75:{$(\ell, \tilde y_\ell)$}}] at (axis cs:1,.8) {};
\node[label={90:{$(u, \tilde y_u)$}},circle,fill,inner sep=2pt] at (axis cs:4,0) {};

\end{axis}

\end{tikzpicture}
\caption{}
\end{subfigure}
\begin{subfigure}[b]{0.23\textwidth}
\centering
\begin{tikzpicture}[scale=0.4, font = \Large,
declare function={  a(\x)  = -20*\x + 80; 
                              b(\x) = 0;
                              c(\x) = 3*\x^3 - 8*\x^2 + 82*\x + 336;
                              yplus(\x) = (-1*b(\x) + sqrt(b(\x)*b(\x) - 4*a(\x)*c(\x)))       /(2*a(\x));
                              yminus(\x) = (-1*b(\x) - sqrt(b(\x)*b(\x) - 4*a(\x)*c(\x)))/(2*a(\x)); } ]
\begin{axis}[no markers,samples=300,xmin=4.25,xmax=11,ymin=-8,ymax=+8]
\path[fill,black!50!white, opacity=0.1]  (axis cs:6,-8) rectangle (axis cs:10,8);

\path [fill,red!80!white, opacity = 0.1] (axis cs:7,-8) -- (axis cs:7,8) -- (axis cs:8,8) -- (axis cs:8,-8);
\tikzvline[color=red,thick]{8}
\tikzvline[color=red,thick]{7}

\addplot[blue,thick, domain=4.25:11] {yplus(x)};
\addplot[blue,thick, domain=4.25:11] {yminus(x)};

\node at (axis cs:5,5) {$y_+$};
\node at (axis cs:5,-5) {$y_-$};

\node[label={60:{$(x^*, y_+(x^*))$}},circle,fill,inner sep=2pt] at (axis cs:7.7882,5.01646) {};
\node[label={300:{$(x^*, y_-(x^*))$}},circle,fill,inner sep=2pt] at (axis cs:7.7882,-5.01646) {};

\addplot[black, mark options={solid}] coordinates {(8,.3) (10,.7)};
\node[circle,fill,inner sep=2pt] at (axis cs:10,.7) {};
\addplot[black, mark options={solid}] coordinates {(7,-2) (6,-1.8)};
\node[label={270:{$(\hat u, \tilde y_{\hat u})$}},circle,fill,inner sep=2pt] at (axis cs:7,-2) {};

\tikzvline[color=black,thick]{-2.84}
\tikzvline[color=black,thick]{3.2}
\end{axis}
\end{tikzpicture}
\caption{}
\end{subfigure}
\begin{subfigure}[b]{0.23\textwidth}
\centering
\begin{tikzpicture}[scale=0.4,font = \Large,
declare function={ a(\x)  = 4*\x - 1; 
                              b(\x) = 32;
                              c(\x) = -2*\x^2 - 16*\x - 32;
                              yplus(\x) = (-1*b(\x) + sqrt(b(\x)*b(\x) - 4*a(\x)*c(\x)))       /(2*a(\x));
                              yminus(\x) = (-1*b(\x) - sqrt(b(\x)*b(\x) - 4*a(\x)*c(\x)))/(2*a(\x)); } ]
\begin{axis}[no markers,samples=100,xmin=1,xmax=10,ymin=-10,ymax=5]
\path[fill,black!50!white, opacity=0.1]  (axis cs:2,-10) rectangle (axis cs:8,5);

\path [fill,red!80!white, opacity = 0.1] (axis cs:2,1) -- (axis cs:8,1.8) -- (axis cs:8,3.1) -- (axis cs:2,2.3);
\draw (axis cs:2,1.65) -- (axis cs:8,2.575);
\draw[dashed,red] (axis cs:2,1) -- (axis cs:8,1.8);
\draw[dashed,red]  (axis cs:2,2.3) -- (axis cs:8,3.1);

\addplot[blue,thick, domain=1:10] {yplus(x)};
\addplot[blue,thick, domain=1:10] {yminus(x)};

\node at (axis cs:9,3.8) {{\large $y_+$}};
\node at (axis cs:9,-5) {$y_-$};

\node[circle,fill,inner sep=2pt] at (axis cs:2,1.652) {};
\node[label={270:{$(\ell, y_+(\ell))$}}] at (axis cs:2.5,1.652) {};
\node[label={270:{$(u, y_+(u))$}},circle,fill,inner sep=2pt] at (axis cs:8,2.57526) {};

\end{axis}
\end{tikzpicture}
\caption{}
\end{subfigure}
\begin{subfigure}[b]{0.23\textwidth}
\centering
\begin{tikzpicture}[scale=0.4,font = \Large]

\begin{axis}[no markers,samples=100,xmin=-4.8,xmax=4.8,ymin=-1.5,ymax=1.5]
\path[fill,black!50!white, opacity=0.1]  (axis cs:-4,-4) rectangle (axis cs:4,4);

\path [fill,red!80!white, opacity = 0.1] (axis cs:-4,0.1) -- (axis cs:4,0.1) -- (axis cs:4,-0.1) -- (axis cs:-4,-0.1);
\draw (axis cs:-4,0) -- (axis cs:4,0);
\draw[dashed,red] (axis cs:-4,0.1) -- (axis cs:4,0.1);
\draw[dashed,red]  (axis cs:-4,-0.1) -- (axis cs:4,-0.1);

\addplot[blue,thick, domain=-5:5] {-1*(x/4+1)*(x/4-1)-1};
\addplot[blue,thick, domain=-5:5] {-1*(x/4+1)*(x/4-1)};

\node at (axis cs:-3,.8) {{ $y_+$}};
\node at (axis cs:-3,-1) {$y_-$};

\node[circle,fill,inner sep=2pt] at (axis cs:-4,0) {};
\node[label={90:{$(\ell, y_+(\ell))$}}] at (axis cs:-3.2,0) {};

\node[circle,fill,inner sep=2pt] at (axis cs:4,0) {};
\node[label={90:{$(u, y_+(u))$}}] at (axis cs:3.5,0) {};

\end{axis}
\end{tikzpicture}
\caption{}
\end{subfigure}

\caption{
The techniques from each subcase of Case~2 from the proof of Lemma~\ref{lem:cubic_y_deg2}. (a) The curve $y^2(8+4x) + y(-32) + (x^3 - 4x^2 - 16x + 32) = 0$  on the region $1 \leq x \leq 4$ satisfies $y_-$ is convex while $y_+$ is concave.  One can find a line separating the two by only considering the $y_+$ and $y_-$ at the endpoints $x=1$ and $x=4$.
(b) The curve $y^2(-20x + 80) + (3x^3 - 8x^2 + 82x + 336) = 0$ on the region $6 \leq x \leq 10$ satisfies $y_+$ is convex while $y_-$ is concave.  We determine the $x$-values $x^*$ where $y_+$ and $y_-$ are closest and then create two separate regions on which to separate $y_+$ from $y_-$.  The separation on each region computes a point and a slope from this point to separate.
(c) The curve $y^2(4x - 1) + y(32) + (2x^2 - 16x - 32)$ on the region $-2 \leq x \leq 8$ satisfies $y_+$ and $y_-$ are concave.   To separate the curves, we draw the line connecting the endpoints of $y_+$ to itself.  The red shaded region around this line in the figure is explained better in the next subfigure.  (d)  A more abstract example of two concave functions shows that connecting the endpoints may still intersect the lower curve.  Therefore, we remove the red shaded region around this line to ensure that we separate the two curves.  
}

\label{fig:cases}
\label{fig:cases-detail}
\end{figure}
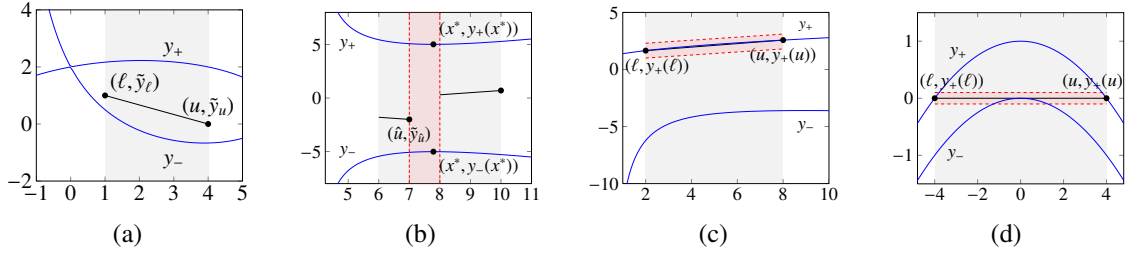

\begin{lemma}
\label{lem:cubic_y_deg2}
Suppose $deg_y(f) = 2$. For any $\omega \in \Z+\tfrac{1}{2}$, we can find a division description for $S_{\leq \omega}^f$ on $B$ in polynomial time.
\end{lemma}
\begin{proof}
We begin by finding a set of $1/4$-approximations $\mathcal A = \{\tilde \alpha_1, \dots, \tilde \alpha_k\}$, $k \leq 1$, of the zeros of $f_2(x)$ with Lemma~\ref{lem:numerical} part~\eqref{part2}. We focus on intervals $I \in \I_{\mathcal A}$, since $f_2(x)$ is non-zero on these intervals.
On the level set $S_{=\omega}^f$, we can write $y$ in terms of $x$ using the quadratic formula, yielding two functions 
$$
y_+(x) = \frac{-f_1(x) + \sqrt{\Delta}}{2f_2(x)},\quad y_-(x) =\frac{-f_1(x) - \sqrt{\Delta}}{2f_2(x)},\quad \text{where}\ \Delta = f_1(x)^2 - 4f_2(x) (f_0(x)-\omega).
$$
By Lemma~\ref{lem:numerical} part~\eqref{part1}, we can test whether or not $\Delta \equiv 0$.

\textbf{\underline{Case 1:} $\Delta \equiv 0$.}  If $\Delta \equiv 0$, then $y_+ \equiv y_-$, meaning that all roots are double roots.  Therefore, $f(x, y) - \omega$ can be written as  $f_2(x) \Big(y - \tfrac{-f_1(x)}{2 f_2(x)}\Big)^2$.  It follows that $\nabla f(x, y_+(x)) = 0$  for all $x$ in the domain of $y_+$, and hence the gradient $\nabla f$ is zero on the level set.  From the definition of affinely critical, Lemma~\ref{lem:cubic-easy} and the fact that $f$ is not constant on $\R^2$, we must have that the level set on $x \in I$ is contained in a line since $y_+$ is differentiable in $I$.
Moreover, we can compute the line exactly by evaluating the derivative and the function at a point where $f_2(x) \neq 0$.  Then we write it as $a x + by = c$ with $a,b,c \in \Z$. As before, our division description comprises lines from $X_{\mathcal A}$ and the line $ax + by = c$, whereas the polyhedra come from $\I_{\mathcal A}$ and the inequalities $ax + by \geq c+1$ and $ax + by \leq c - 1$.

\textbf{\underline{Case 2:} $\Delta \not\equiv 0$.}
By Lemma~\ref{lem:quadratics}, we can find a list $\mathcal A = \{\tilde \alpha_1, \dots, \tilde \alpha_k\}$ of rational points such that $y_\pm$ are well defined, continuous and either convex or concave on each $I \in \I_{\mathcal A}$. Moreover, on each interval $\Delta \ne 0$, so they do not intersect. Hence, $y_\pm$ are convex or concave (or both) on each interval.
Furthermore, we can determine whether $y_+ > y_-$ or $y_+ < y_-$ on the interval by evaluating one point in the interval.

We will assume from here on that $y_+ > y_-$ on the interval $I$ as the calculations are similar if $y_- < y_+$.  Note that $y_+ > y_-$ on $I$ implies that $f_2 > 0$ on $I$.
Let $\ell,u \in \Z$ be the endpoints of $I$, that is $I = [\ell, u]$.
Since we are interested in a division description on $B$, we may assume $-R \le \ell \le u \le R$.  We distinguish the following four cases based on the convexity or concavity of $y_+, y_-$ on $I$.

\textbf{Case 2a: $y_+$ concave, $y_-$ convex.} (cf.\  Figure~\ref{fig:cases} (a)) Consider $f(\ell,y) - \omega$ and $f(u,y) - \omega$ as quadratic polynomials in $y$. We use Lemma~\ref{lem:numerical} part~\eqref{part3} to find upper and lower bounds on their roots.
Since Lemma~\ref{lem:numerical} part~\eqref{part3} finds non-intersecting bounding boxes on each root for any prescribed $\epsilon$, we simply take $\epsilon = 1$.  These approximations are actually approximations to the values of $y_-(x)$ and $y_+(x)$ at $x=\ell$ and $x=u$.  Take the averages between the lower bounds of the upper roots and the upper bounds of the lower roots, and call these averages $\tilde y_\ell$ and $\tilde y_u$.  
Consider the rational line segment $\conv\{ (\ell, \tilde y_\ell), (u, \tilde y_u)\}$.  Due to the convexity and concavity of $y_+$ and $y_-$ on the interval, this line segment separates $y_-$ from $y_+$ on $[\ell,u]$.

\textbf{Case 2b: $y_+$ convex, $y_-$ concave.} (cf.\  Figure~\ref{fig:cases} (b)) Since the epigraph of $y_+$ and the hypograph of $y_-$ are both convex sets, there exists a hyperplane that separates them due to the hyperplane separation theorem. To find such a hyperplane, suppose first that we can exactly determine $x^*$ that minimizes $y_+(x) - y_-(x)$ on $[\ell, u]$, and suppose further that we could exactly compute $y'_+(x^*)$.
If $x^*\in (\ell, u)$, then $y'_+(x^*) = y'_-(x^*)$, so the line passing through $x^*$ with slope $y'_+(x^*)$ separates the two regions. Otherwise, suppose that $x^* = \ell$. Then $y'_-(x^*) \le y'_+(x^*)$, so the same line again separates the two regions. The case where $x^* = u$ is analogous.

However, $x^*$ may be irrational, so we might not be able to determine it exactly. To find a numerical approximation $\tilde x^*$, note that $y_+ - y_- = \sqrt{\Delta}/f_2$.  Since this quantity is nonnegative on $[\ell, u]$, we instead minimize the square, which is $\Delta/f_2^2$.  This is a quotient of polynomials, and therefore we can approximately compute the zeros of the first derivative, which occur at $\Delta' f_2^2 - 2 f_2^{} f'_2 \Delta = 0$.  In fact, either $y_-$, $y_+$ are both lines, or there are at most polynomially many local minima.  

Let $\mathcal B = \{\tilde \beta_1, \dots, \tilde \beta_{\hat k}\}$ be the $\epsilon$-approximations of these roots with $\epsilon = 1/4$. We consider $\I_{\mathcal A \cup \mathcal B}$ and $X_{\mathcal A \cup \mathcal B}$ and we focus on an interval $\hat I \in \I_{\mathcal A \cup \mathcal B}$ with $\hat I \subseteq I$.  Let $\hat \ell, \hat u \in \Z$ be the endpoints of $\hat I$.  Since $\hat I \in \I_{\mathcal A \cup \mathcal B}$, no minimizer of $y_+ - y_-$ lies in $(\hat \ell, \hat u)$.

Since no minimizer lies in $(\hat \ell, \hat u)$, $y_+(x) - y_-(x)$ is minimized either at $\hat \ell$ or at $\hat u$, so we just compare the values.  Since $y_+ - y_- = \sqrt{\Delta}/f_2$ is nonnegative on $I$, we instead compare the squares $(y_+(\hat \ell) - y_-(\hat \ell))^2$ and $(y_+(\hat u) - y_-(\hat u))^2$, thus avoiding approximation of square roots.  Suppose without loss of generality that $\hat u$ is the minimizer. 
 
Now consider 
$$
 y_+' - y_-' = (\sqrt{\Delta}/f_2^{})' = \frac{(\Delta' f_2^{} - 2 f'_2 \Delta)\sqrt{\Delta}}{2 f_2^2 \Delta}.
$$  
Call $a = \Delta' f_2^{} - 2 f'_2 \Delta$ and $b = 2 f_2^2$. Then $a,b,\Delta \colon \Z\cap \hat I \to \Z \setminus\{0\}$, and they are all polynomials of bounded degree.  A straightforward calculation shows that  $1 \leq |a(\hat u)|, |b(\hat u)|, |\Delta(\hat u)| \leq 4 \times 10^2 M^3 R^5$.  Therefore
$$
 |y_+'(\hat u) - y_-'(\hat u)| = \left| \frac{a(\hat u)}{b(\hat u) \sqrt{\Delta(\hat u)}} \right| \geq \left| \frac{1}{\left( 4 \times 10^2 M^3 R^5 \right)^2} \right|.
$$
Let $\epsilon := \frac{1}{4 \left( 4 \times 10^2 M^3 R^5 \right)^2}$.  We need to approximate $y_+'(\hat u)$ and $y_-'(\hat u)$ within a factor of $\epsilon$.
Using the representation in Lemma~\ref{lem:quadratics}, we have 
$$
y'_+ =  \frac{\Delta   \big( f_2' f_1^{} -f_2^{}  f_1' \big)+ (-\Delta   f_2' +\frac{1}{2} f_2^{}  \Delta')\sqrt{\Delta} }{2 f_2^2\Delta} = \frac{X + Y \sqrt{\Delta}}{Z} ,
$$
Hence, we can compute $X,Y,Z$ exactly, but we need to approximate $\sqrt{\Delta}$.  A straightforward calculation shows that we only need to approximate this within a factor of $\hat \epsilon := \tfrac{\epsilon}{4 \times 10^2 M^3 R^5}$.  
A similar calculation follows for $y'_-$.

We can compute $\epsilon$-approximations $\tilde y_+(\hat u)$ and $\tilde y_-(\hat u)$ using a numerical square root tool such as~\cite{mansour_complexity_1989} to approximate $\sqrt{\Delta(\hat u)}$ to an accuracy of $\hat \epsilon$.
Let 
$$
m = \tfrac{1}{2}(\tilde y_+'(\hat u) + \tilde y_-'(\hat u)) \in [\min(y_-(\hat u), y_+(\hat u)), \max(y_-(\hat u), y_+(\hat u))].
$$
Moreover, we compute $\tilde y_{\hat u} = \tfrac{1}{2}(\tilde y_+(\hat u) + \tilde y_-(\hat u))$ where $\tilde y_+(\hat u)$ and $\tilde y_-(\hat u)$ are approximations computed from the roots of $f(\hat u,y)$ using Lemma~\ref{lem:numerical}.  
Then the line through $(\hat u, \tilde y_{\hat u})$ with slope $m$ separates $y_-$ and $y_+$ on $[\hat \ell, \hat u]$.

\textbf{Case 2c: $y_+$ concave, $y_-$ concave.} (cf.\ 
Figure~\ref{fig:cases} (c) and (d)) Consider the line segment $L$ connecting the two endpoints of $y_+$, i.e., $\conv\{(\ell, y_+(\ell)), (u, y_+(u))\}$.  
We claim that $L$ intersects the graph of $y_-$ in at most one point in $[\ell, u]$. By Remark~\ref{rem:three-intersections}, either $L$ coincides with the graph of $y_+$, or $L$ intersects the level set $S^f_{=\omega}$ at most three times. Since $L$ intersects the graph of $y_+$ twice and $y_-(x) < y_+(x)$, $L$ can intersect the graph of $y_-$ at most once.

Therefore, the line $L$ is a weak separator of the curves $y_-(x)$ and $y_+(x)$.  Since $L$ may be irrational, we cannot compute it exactly, but we approximate it instead. 
Let $\epsilon := \frac{1}{2 (u - \ell)}$.  We use Lemma~\ref{lem:numerical} part~\eqref{part3} to find bounding $\epsilon$-approximations to the roots of the equations $f(\ell, y) = \omega$ and $f(u,y) = \omega$.  Hence we can obtain $\tilde y_\ell^1 < y_+(\ell) <   \tilde y_\ell^2$ and $\tilde y_u^1 < y_+(u) <   \tilde y_u^2$ such that $|\tilde y_\ell^1 -  \tilde y_\ell^2| < \epsilon$ and $|\tilde y_\ell^1 -   \tilde y_\ell^2| < \epsilon$.  We then construct the quadrangle $Q = \conv(\{ (\ell, \tilde y_\ell^1), (\ell, \tilde y_\ell^2), (u, \tilde y_u^1), (u, \tilde y_u^2) \})$.  By construction, $\vol(Q) \leq \epsilon (u - \ell) \leq 1/2$. Therefore, by Lemma~\ref{lem:line}, this contains at most one line of integer points that we can compute a description for in polynomial time. We add this line to our division description.

Furthermore, $L \subseteq Q$.  Therefore, $y_-$ is strictly below the line $\conv\{(\ell, \tilde y_\ell^2) ,(u, \tilde y_u^2)\}$ and $y_+$ is strictly above the line $\conv\{(\ell, \tilde y_\ell^1), (u, \tilde y_u^1)\}$.  We then add to our division description the two polyhedra given by $(x, y) \in [\ell,u] \times \R$ such that $(x, y)$ is either above $\conv\{(\ell, \tilde y_\ell^2) ,(u, \tilde y_u^2)\}$ or below $\conv\{(\ell, \tilde y_\ell^1), (u, \tilde y_u^1)\}$.

\textbf{Case 2d: $y_+$ convex, $y_-$ convex.} This case is analogous to the previous case, where instead here we take the line segment $\conv\{(\ell, y_-(\ell)), (u, y_-(u))\}$.

We have shown how to divide each interval, thus completing the proof.
\end{proof}

\begin{lemma}
\label{lem:cubic_y_deg3}
Suppose $deg_y(f) = 3$.  For any $\omega \in \Z+\tfrac{1}{2}$, we can find a division description for $S_{\leq \omega}^f$ on $B$ in polynomial time.
\end{lemma}
\begin{proof}
We create the division description by applying a linear transformation such that the objective function becomes a quadratic function in one variable and then apply Lemma~\ref{lem:cubic_y_deg2}.  For any $a \in \R$, consider the linear transformation  $x = u$ and $y = au + v$, that is $A(u,v) = (x,y)$ where 
$$
A = 
\begin{bmatrix}
1 & 0 \\
a & 1
\end{bmatrix}.
$$
Notice that $A$ is invertible, and $u = x$, $v = y - ax$.
Define
$$
g_a(u,v) := f(u, au + v) = (c_0 + c_1 a + c_2 a^2 + c_3 a^3) u^3 + w_a(u,v),
$$
where $w_a(u,v)$ is at most quadratic in terms of $u$ and at most cubic in terms of $v$, and $c_0, \dots, c_3$ are a subset of the coefficients of $f$.  Let $\bar a \in \R$ be such that 
$$
c_0 + c_1 \bar a + c_2 \bar a^2 + c_3 \bar a^3 = 0.
$$
Note that since $\deg_y(f) = 3$, we have that $c_3 \ne 0$. Since this is a cubic equation with integer coefficients, we know there is at least one real solution.
Define $\bar A = \begin{bmatrix} 1 & 0 \\ \bar a & 1 \end{bmatrix}$ and let $\bar R \ge 1$ be an upper bound on $|\bar a| + 1 = \| \bar A^{-1} \|_1 = \| \bar A^{-1} \|_\infty$, which can be chosen of polynomial size in terms of the coefficients of $f$ by Theorem~\ref{thm:polynomial_root_bounds}.
Set $1/\epsilon = 4 \times 36 \times M (2 \bar R + 1)^3 (\bar R + 1)^3 R^3$ and compute an approximation $\bar a_\epsilon$ such that $|\bar a - \bar a_\epsilon| < \epsilon$.  By Lemma~\ref{lem:numerical} part~\eqref{part2}, this approximation can be found in polynomial time.

Since $\epsilon \le 1$, for $1 \le i \le 3$ we have 
$
| \bar a^i - \bar a_\epsilon^i | \le \epsilon \ i \ (2 \bar R + 1)^i
$.
Thus for any $(u,v)$ with $\|(u,v)\|_2 \leq (\bar R + 1) R$, we have
\begin{align*}
|g_{\bar a_\epsilon}(u,v) - w_{\bar a_\epsilon}(u,v)| 
&= \left| (c_0 + c_1 \bar a_\epsilon + c_2 \bar a_\epsilon^2 + c_3 \bar a_\epsilon^3) \ u^3 \right|
\le \left| \sum_{i=1}^3 c_i (\bar a^i - \bar a_\epsilon^i) \right| (\bar R + 1)^3 R^3
\le \sum_{i=1}^3 |c_i| |\bar a^i - \bar a_\epsilon^i| (\bar R + 1)^3 R^3
\\
&\le \sum_{i=1}^3 4 \ M \epsilon \ 3 \ (2 \bar R + 1)^3 (\bar R + 1)^3 R^3
\le \epsilon \times 36 \times M (2 \bar R + 1)^3 (\bar R + 1)^3 R^3
\le \tfrac{1}{4} .
\end{align*}
Let $f_\epsilon(x,y) := w_{\bar a_\epsilon}(x, y - \bar a_\epsilon x)$ and consider $\bar A_\epsilon (u,v) = (x,y)$ with $\bar A_\epsilon = \begin{bmatrix} 1 & 0 \\ \bar a_\epsilon & 1 \end{bmatrix}$. Then, for all $x,y \in B$ we have
$$
\|(u,v)\|_2 = \|\bar A_\epsilon^{-1} (x,y)\|_2
\leq \|\bar A_\epsilon^{-1} \|_2 \, \| (x,y)\|_2
\leq \|\bar A_\epsilon^{-1} \|_1 \, \| (x,y)\|_2
\leq (\bar R + 1) \, \|(x,y)\|_2 \le (\bar R + 1) R ,
$$
because $\|\bar A_\epsilon^{-1} \|_2 \le \sqrt{\|\bar A_\epsilon^{-1} \|_1 \|\bar A_\epsilon^{-1} \|_\infty} = \|\bar A_\epsilon^{-1} \|_1$ and $|\bar a - \bar a_\epsilon| < \epsilon \le 1$. Thus
$$
|f(x,y) - f_\epsilon(x,y)| = |g_{\bar a_\epsilon}(u,v) - w_{\bar a_\epsilon}(u,v)| \leq 1/4 .
$$
Thus, for $\omega \in \Z+\tfrac{1}{2}$, we have that 
$$
\{ (x,y) \in B \cap \Z^2 : f(x,y) \leq \omega\} = \{ (x,y) \in B \cap \Z^2 : f_\epsilon(x,y) \leq \omega\}.
$$

Therefore, we can solve the feasibility problem for $f$ by solving the feasibility problem for $f_\epsilon$.
By Lemma~\ref{lem:cubic_y_deg2}, we can find $P_i, Q_i$ and $L_i$ to divide the plane for the level sets of $w_{\bar a_\epsilon}$, which is done by scaling the function to have integer coefficients.  Then, under the linear transformation $\bar A_\epsilon$, we have that $\bar A_\epsilon P_i, \bar A_\epsilon Q_i, \bar A_\epsilon L_i$ are all rational polyhedra, so they comprise a division description.
\end{proof}

\begin{theorem}[cubic, bounded]
Theorem~\ref{thm:min_cubic_unbounded} holds when $P$ is bounded.
\label{thm:cubic}
\end{theorem}
\begin{proof}
Follows directly from Lemmas~\ref{lem:cubic_y_deg0},~\ref{lem:cubic_y_deg1},~\ref{lem:cubic_y_deg2},~\ref{lem:cubic_y_deg3} and Theorem~\ref{thm:dd-min}.
\end{proof}

\section{Homogeneous Polynomials}
\label{sec:homogeneous}
In this section we will prove Theorem~\ref{thm:min_homogeneous_bounded} by showing that for homogeneous polynomials, we can choose one division description of the plane that works for all level sets.  This is done by appropriately approximating the regions where the function is quasiconvex and quasiconcave.  We say that a function $f$ is quasiconvex on a convex set $S$ if all the sublevel sets are convex, i.e., $\{ \x \in S : f(\x) \leq \omega\}$ is convex for all $\omega \in \R$.  A function $f$ is quasiconcave on a set $S$ if $-f$ is quasiconvex on $S$.
A polyhedral division of the regions of quasiconvexity and quasiconcavity is a division description for any sublevel set.
Therefore, if we can divide the domain into polyhedral regions where the objective is either quasiconvex or quasiconcave, then we can apply Theorem~\ref{thm:dd-min}.

In Section~\ref{sec:homogeneous-bordered-hessian} we study homogeneous functions and investigate where they are quasiconvex or quasiconcave.  In Section~\ref{sec:homogeneous-numerical}, we show how to divide the domain appropriately into regions of quasiconvexity and quasiconcavity, proving the main result for homogeneous polynomials.

\subsection{Homogeneous Functions and the Bordered Hessian}
\label{sec:homogeneous-bordered-hessian}
The main tool that we will use for distinguishing regions where $f$ is quasiconvex or quasiconcave is the \emph{bordered Hessian}.  For a twice differentiable function $f \colon \R^n \to \R$ the bordered Hessian is defined as 
\begin{equation}
H_f = 
\begin{bmatrix}
0 & \nabla f^T\\
\nabla f  & \nabla^2 f
\end{bmatrix} ,
\end{equation}
where $\nabla f$ is the gradient of $f$ and $\nabla^2 f$ is the Hessian of $f$.

We will denote by $D_f$ the determinant of the bordered Hessian of $f$.
Let $f_i$ denote the partial derivative of $f$ with respect to $x_i$ and $f_{ij}$ denote the mixed partial with respect to $x_i$ and $x_j$. 
The following result can be derived from Theorem~2.2.12 and 3.4.13 in \cite{cambini_generalized_2009}.
\begin{lemma}
\label{lem:determine-quasiconvex}
Let $f \colon \R^2 \to \R$ be a continuous function that is twice continuously differentiable on a convex set $S$. 
\begin{enumerate}[(i)]
\item If $D_f < 0$, then $f$ is quasiconvex on the closure of $S$.
\label{itm:Dfsmaller0}
\item If $D_f > 0$, then $f$ is quasiconcave on the closure of $S$.
\label{itm:Dflarger0}
\end{enumerate}

\end{lemma}

We now briefly discuss general homogeneous functions.  
We say that $h\colon \R^n \to \R$ is homogeneous of degree $d$ if $h(\lambda \x) = \lambda^d h(\x)$ for all $\x \in \R^n$ and $\lambda \in \R$.  Clearly a homogeneous polynomial of degree $d$ is a homogeneous function of degree $d$.  

The following lemma shows that the determinant of the bordered Hessian has a nice formula for any homogeneous function.  This was proved in Hemmer \cite{hemmer_limiting_1995} for the case of $n=2$, which can be adapted easily to general $n$.

\begin{lemma}
\label{lemma:Df-det}
Let $h\colon \R^n \to \R$ be a twice continuously differentiable homogeneous function of degree $d \geq 2$.  Then 
\begin{equation}
\label{eq:Df-det}
D_h(\x) = \frac{-d}{d-1} h(\x) \cdot \det(\nabla^2h(\x)) \ \text{ for all } \x \in \R^n.
\end{equation}
\end{lemma}

Recall that a polynomial $f(\x)\colon \R^n \to \R$ is homogeneous translatable if there exists a $\t \in \R^n$ such that $f(\x+\t) = h(\x)$ for some homogeneous polynomial $h(\x)$.
\begin{corollary}
\label{cor:Df_homogeneous}
Let $f$ be a homogeneous translatable polynomial in $2$ variables of degree $d \ge 2$.  Then either $D_f$ is identically equal to zero, or $D_f$ is a homogeneous translatable polynomial that translates to a homogeneous polynomial of degree $3d - 4$.
\end{corollary}

If $h\colon \R^2 \to \R$ is a homogeneous polynomial of degree $d \ge 2$ and $D_h \equiv 0$, then $h$ is the power of a linear form, as the following lemma shows.

\begin{lemma}[Lemma~3 in \cite{hemmer_limiting_1995}]
\label{lem:D_f=0}
Let $h\colon \R^2 \to \R$ be a homogeneous polynomial of degree $d \ge 2$. Then $D_h \equiv 0$ if and only if there exist $\ve c \in \R^2$ such that 
\begin{equation}
h(\x) = ({\ve c}^T \x)^d .
\end{equation}
\end{lemma}

\subsection{Division of Quasiconvex and Quasiconcave Regions and Proof of Theorem~\ref{thm:min_homogeneous_bounded}}
\label{sec:homogeneous-numerical}

Let $P$ be a bounded rational polyhedron and let $f$ be a homogeneous translatable polynomial and suppose that $D_f$ is also homogeneous translatable.  Recall that there exists an integer $R \ge 1$ whose size is polynomial in the size of $P$ such that $P \subseteq B := [-R, R]^2$.

We will show how to decompose $B \cap \Z^2$ into polyhedra $P_i$ where $D_f < 0$, $Q_i$ where $D_f > 0$, and lines $L_k$.  Thus we obtain a classification of regions of quasiconvexity and quasiconcavity by Lemma~\ref{lem:determine-quasiconvex} and can then use Theorem~\ref{thm:dd-min}.

The regions $D_f \leq 0$ and $D_f \geq 0$ cannot be described by rational hyperplanes; therefore, we approximate them sufficiently closely by rational hyperplanes.  In order to avoid numerical difficulties, we allow the possibility of leaving out a line of integer points, which we consider separately.  To determine those lines, Lemma~\ref{lem:line} will be useful.

We will summarize a strategy to create the desired regions for a homogeneous polynomial $h$ of degree $d$.  We will then prove a theorem for the more general setting of a homogeneous translatable polynomial $F$ in a similar way, which we later apply to $F = D_f$.

For a homogeneous polynomial $h$ of degree $d$ that is not the zero function, the roots of $h$ must lie on at most $d$ lines, which we will call zero lines.  This is because if $h(\bar \x) = 0$, then we have $h(\lambda \bar \x) = \lambda^d h(\bar \x) = 0$ for all $\lambda \in \R$.  Each zero line is either the line $x_1 = 0$, or must intersect the line $x_2 = \delta$ for any fixed $\delta \neq 0$.  The fact that the former is a zero line can be established by testing whether $h(0, x_2)$ is the zero polynomial.  To see how the latter defines zero lines, consider the polynomial $h(x_1, \delta)$. It is a univariate polynomial of degree $d$, and hence has at most $d$ roots.  Therefore, finding these roots (and hence the intersections of the zero lines of $h$ with the line $x_2 = \delta$) completes our classification of all the zero lines of $h$.

We choose $\delta = R$ and use Lemma~\ref{lem:numerical} part~\eqref{part3} to find intervals containing the roots of $h(x_1, R)$.
These intervals can be used to create quadrilaterals that cover the zero lines of $h$ within $B$.  Provided that the proper accuracy is used, these quadrilaterals will contain at most one line of integer points.
Finally, taking the complement of quadrilaterals in $B$, we find a union of polyhedra where $h$ is non-zero.  We can then find an interior point of each polyhedron on which we can evaluate $h$ to determine the sign on each polyhedron.

\begin{theorem}
\label{thm:partition_Z2-translatable}
Let $F(\x)$ be a homogeneous translatable polynomial in two variables of degree $d \in \Z_{\geq 0}$ with integer coefficients that is not the zero function. Let $\ell \in \Z_+$ be a bound on the size of the coefficients of $F$.  In polynomial time in $\ell$, $d$ and the size of $R$, we can partition $B \cap \Z^2$ into three types of regions:
\begin{enumerate}[(i)]
\item rational polyhedra $P_i$ where $F(\x) > 0$ for all $\x \in P_i \cap B \cap \Z^2$,  
\item rational polyhedra $Q_j$ where $F(\x) < 0$ for all $\x \in Q_j \cap B \cap \Z^2$,
\item one dimensional rational linear spaces $L_k$,
\end{enumerate}
where each polyhedron $P_i, Q_i$ is described by polynomially many rational linear inequalities and each linear space is described by one rational hyperplane, and all have size polynomial in $\ell$, $d$, and the size of $R$.  Furthermore, there are only polynomially many polyhedra $P_i, Q_i$ and linear spaces $L_k$.
\end{theorem}
\begin{proof}
Since $F$ is homogeneous translatable, there exists 
$\t \in \R^2$ such that $h(\x) = F(\x + \t)$ is a homogeneous polynomial of degree at most $d$.  The zeros of $h$ lie on lines through the origin.  Therefore, the zeros of $F$ must lie on lines through $\t$.  

We consider the region $S = \{\x \in \R^2 : -R \leq x_2 \leq R\}$. Since all zero lines of $F$ pass through $\t$, the geometry of the nonnegative regions in $B$ depends on whether or not $\t \in S$.

Consider the distinct roots $\alpha_1< \alpha_2 <\dots <  \alpha_r$ and $\beta_1 < \beta_2 < \dots < \beta_s$ of the univariate polynomials $F(x_1, R)$ and $F(x_1, -R)$, respectively.  Assume without loss of generality that $r = s$, because if we encounter the case $r \ne s$, then $\t$ is on the boundary of the region, so we can simply set $R := R + 1$, which will result in $r = s$.

If $\t \notin S$, then the zero lines of $F$ do not intersect in $S$ and hence there must be a zero line from each $\alpha_i$ to each $\beta_i$ (cf.\ Figure~\ref{fig:homogeneous-zeros-translatable}a). If $\t \in S$, the zero lines intersect in $S$, and therefore the zero lines of $F$ connect each $\alpha_i$ to each $\beta_{r-i}$ (cf.\ Figure~\ref{fig:homogeneous-zeros-translatable}b). The union of those two sets of lines has cardinality at most $2d$ and contains all zero lines of $F$ except for lines parallel to the $x_2$ axis (cf.\ Figure~\ref{fig:homogeneous-zeros-translatable}c). By switching $x_1$ and $x_2$, repeating the same procedure and adding the resulting at most $2d$ lines, we thus get a set of at most $4d$ lines which contain all the zero lines of $F$.

For each of the at most $4 d$ lines we now construct a rational quadrilateral containing its intersection with $B$ and with the property that all integer points in the quadrilateral are contained on a single line. We describe this only where we fix $x_2 = \pm R$, as the case of fixing $x_1$ is similar. Consider the line passing through $(\alpha, R)$ and $(\beta, -R)$. In time and output size that is bounded by polynomial in $\ell$, and the size of $R$ and $\epsilon$ (see Lemma~\ref{lem:numerical}), we can compute a sequence of disjoint intervals, each of length smaller than $\epsilon$ containing the roots of $F(x_1, \pm R)$.  Thus, for the root $\alpha$, we can construct $\alpha^-, \alpha^+ \in \Q$  with $\alpha^- < \alpha < \alpha^+$ such that $|\alpha^+ - \alpha^-| < \epsilon$. Similarly, we can construct $\beta^-, \beta^+ \in \Q$ for $\beta$.

If we choose $\epsilon := \tfrac{1}{4R}$, then the rational quadrilateral defined by vertices $(\alpha^-, R)$, $(\alpha^+, R)$, $(\beta^-, -R)$ and $(\beta^+, -R)$ has volume less than $2\times R \tfrac{1}{4R} = \tfrac{1}{2}$. Moreover, it can be defined by inequalities with a polynomial size description. Hence by Lemma~\ref{lem:line}, the integer points in it are contained in a rational line that we can find a description of in polynomial time. Each one of those lines will define one $L_k$ in the division description.

In total, this results in at most $4d$ linear spaces $L_k$, $8d$ hyperplanes for the quadrilaterals and $4$ hyperplanes for the boundary of $B$. Apply the hyperplane arrangement algorithm of \cite{sleumer_output_1998} to enumerate all $O(d^2)$ cells of the arrangement in $O(d^3 \, \ell p(2,d))$ time. Here $\ell p(2,d)$ is the cost of running a linear program in dimension 2 with $d$ inequalities, which is in polynomial time because of the inputs are of polynomial size.  
For each cell, a signed vector of $+$, $0$, and $-$ describing if the relatively open cell satisfies $\ve a^i\cdot  \x > b_i$, $\ve a^i\cdot \x = b_i$, $\ve a^i\cdot \x < b_i$, respectively, for every hyperplane $\ve a_i\cdot \x = b_i$ in the arrangement. 
We exclude cells that are contained in the union of quadrilaterals and cells not contained in $B$ by reading these signed vectors.  This is all done in polynomial time in $d$.

We have described how to obtain at most $4d$ linear spaces $L_k$ and polynomially many rational polyhedra that contain all integer points in $B$.  Using linear programming techniques, we can determine an interior point of each of these polyhedra.  Evaluating $F$ at each interior point determines the sign of $F$ on each polyhedron. Hence this construction determines the list of polyhedra $P_i$ and $Q_j$. This finishes the result.
\end{proof}

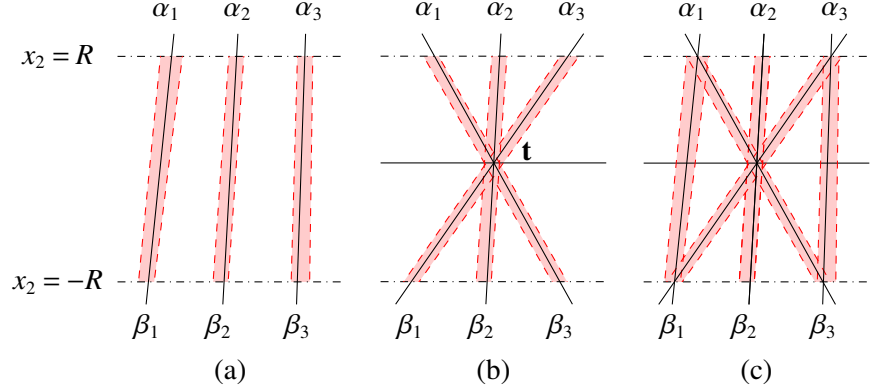
\begin{figure}
\centering
\begin{tikzpicture}[scale = 1]
\begin{scope}
\clip (0,0) rectangle (3,3);

\path [fill,red!20!white] (0.6,3.3) -- (0.25,-0.3) -- (0.45,-0.3) -- (0.9,3.3);
\path [fill,red!20!white] (1.5,3.3) -- (1.25,-0.3) -- (1.45,-0.3) -- (1.7,3.3);
\path [fill,red!20!white] (2.4,3.3) -- (2.3,-0.3) -- (2.55,-0.3) -- (2.6,3.3);

\end{scope}

\draw[dashdotted] (0,3) -- (3,3);
\draw[dashdotted] (0,0) -- (3,0);

\begin{scope}
\clip (0,0) rectangle (3,3);

\draw[dashed,red] (0.6,3.3) -- (0.25,-0.3);
\draw[dashed,red] (0.9,3.3) -- (0.45,-0.3);
\draw[dashed,red] (1.5,3.3) -- (1.25,-0.3);
\draw[dashed,red] (1.7,3.3) -- (1.45,-0.3);
\draw[dashed,red] (2.4,3.3) -- (2.3,-0.3);
\draw[dashed,red] (2.6,3.3) -- (2.55,-0.3);

\end{scope}

\draw (0.75,3.3) -- (0.38,-0.3);
\draw (1.6,3.3) -- (1.4,-0.3);
\draw (2.5,3.3) -- (2.38,-0.3);

\node at (0.65,3.6) {$\alpha_1$};
\node at (1.6,3.6) {$\alpha_2$};
\node at (2.55,3.6) {$\alpha_3$};
\node at (0.4,-0.6) {$\beta_1$};
\node at (1.35,-0.6) {$\beta_2$};
\node at (2.4,-0.6) {$\beta_3$};

\node at (-0.8,3) {$x_2 = R$};
\node at (-0.8,0) {$x_2 = -R$};

\node at (1.5,-1.2) {(a)};
\end{tikzpicture}
\quad
\begin{tikzpicture}[scale= 1]
\begin{scope}
\clip (0,0) rectangle (3,3);

\path [fill,red!20!white] (0.4,3.3) -- (2.45,-0.3) -- (2.65,-0.3) -- (0.6,3.3);
\path [fill,red!20!white] (1.5,3.3) -- (1.25,-0.3) -- (1.45,-0.3) -- (1.7,3.3);
\path [fill,red!20!white] (2.6,3.3) -- (0.1,-0.3) -- (0.3,-0.3) -- (2.85,3.3);

\end{scope}

\draw[dashdotted] (0,3) -- (3,3);
\draw[dashdotted] (0,0) -- (3,0);

\draw (0,1.58) -- (3,1.58);

\begin{scope}
\clip (0,0) rectangle (3,3);

\draw[dashed,red] (0.4,3.3) -- (2.45,-0.3);
\draw[dashed,red] (0.6,3.3) -- (2.65,-0.3);
\draw[dashed,red] (1.5,3.3) -- (1.25,-0.3);
\draw[dashed,red] (1.7,3.3) -- (1.45,-0.3);
\draw[dashed,red] (2.6,3.3) -- (0.1,-0.3);
\draw[dashed,red] (2.85,3.3) -- (0.3,-0.3);

\end{scope}

\draw (0.55,3.3) -- (2.55,-0.3);
\draw (1.6,3.3) -- (1.4,-0.3);
\draw (2.7,3.3) -- (0.2,-0.3);

\node at (0.65,3.6) {$\alpha_1$};
\node at (1.6,3.6) {$\alpha_2$};
\node at (2.55,3.6) {$\alpha_3$};
\node at (0.4,-0.6) {$\beta_1$};
\node at (1.35,-0.6) {$\beta_2$};
\node at (2.4,-0.6) {$\beta_3$};

\node at (1.94,1.74) {$\t$};

\node at (1.5,-1.2) {(b)};
\end{tikzpicture}
\quad
\begin{tikzpicture}[scale = 1]
\begin{scope}
\clip (0,0) rectangle (3,3);

\path [fill,red!20!white] (0.6,3.3) -- (0.25,-0.3) -- (0.45,-0.3) -- (0.9,3.3);
\path [fill,red!20!white] (1.5,3.3) -- (1.25,-0.3) -- (1.45,-0.3) -- (1.7,3.3);
\path [fill,red!20!white] (2.4,3.3) -- (2.3,-0.3) -- (2.55,-0.3) -- (2.6,3.3);

\path [fill,red!20!white] (0.4,3.3) -- (2.45,-0.3) -- (2.65,-0.3) -- (0.6,3.3);
\path [fill,red!20!white] (1.5,3.3) -- (1.25,-0.3) -- (1.45,-0.3) -- (1.7,3.3);
\path [fill,red!20!white] (2.6,3.3) -- (0.1,-0.3) -- (0.3,-0.3) -- (2.85,3.3);

\end{scope}

\draw[dashdotted] (0,3) -- (3,3);
\draw[dashdotted] (0,0) -- (3,0);

\draw (0,1.58) -- (3,1.58);

\begin{scope}
\clip (0,0) rectangle (3,3);

\draw[dashed,red] (0.4,3.3) -- (2.45,-0.3);
\draw[dashed,red] (0.6,3.3) -- (2.65,-0.3);
\draw[dashed,red] (1.5,3.3) -- (1.25,-0.3);
\draw[dashed,red] (1.7,3.3) -- (1.45,-0.3);
\draw[dashed,red] (2.6,3.3) -- (0.1,-0.3);
\draw[dashed,red] (2.85,3.3) -- (0.3,-0.3);

\draw[dashed,red] (0.6,3.3) -- (0.25,-0.3);
\draw[dashed,red] (0.9,3.3) -- (0.45,-0.3);
\draw[dashed,red] (1.5,3.3) -- (1.25,-0.3);
\draw[dashed,red] (1.7,3.3) -- (1.45,-0.3);
\draw[dashed,red] (2.4,3.3) -- (2.3,-0.3);
\draw[dashed,red] (2.6,3.3) -- (2.55,-0.3);

\end{scope}

\draw (0.55,3.3) -- (2.55,-0.3);
\draw (1.6,3.3) -- (1.4,-0.3);
\draw (2.7,3.3) -- (0.2,-0.3);

\draw (0.75,3.3) -- (0.38,-0.3);
\draw (1.6,3.3) -- (1.4,-0.3);
\draw (2.5,3.3) -- (2.38,-0.3);

\node at (0.65,3.6) {$\alpha_1$};
\node at (1.6,3.6) {$\alpha_2$};
\node at (2.55,3.6) {$\alpha_3$};
\node at (0.4,-0.6) {$\beta_1$};
\node at (1.35,-0.6) {$\beta_2$};
\node at (2.4,-0.6) {$\beta_3$};

\node at (1.5,-1.2) {(c)};
\end{tikzpicture}

\caption{
This figure illustrates the techniques in Theorem~\ref{thm:partition_Z2-translatable}. The black solid lines are the zero lines of $F$. We construct the shaded quadrilaterals using numerical approximations on the roots $\alpha_i$ and $\beta_i$ to an accuracy such that they each contain at most one line of integer points.  (a) If the zero lines of $F$ do not intersect in $S$, then each zero line passes through $(\alpha_i, R)$, $(\beta_i, -R)$ for some $i$. (b) If the zero lines of $F$ intersect in $S$, then they intersect in a common intersection point $\t$. Each zero line passes through $(\alpha_i, R)$, $(\beta_{r-i}, -R)$ for some $i$, except for a potential horizontal zero line. (c) To avoid having to know whether the zero lines intersect in $S$ or not, we simply consider all potential lines from both cases. Note that this still does not include a potential horizontal zero line.  This line is covered by switching $x_1$ with $x_2$ and repeating the same procedure.
}
\label{fig:homogeneous-zeros-translatable}
\end{figure}

Recall that for a convex set $C$, $f$ is quasiconvex on $C$ if and only if $C \cap S_{\leq \omega}^f$ is convex $\forall \, \omega \in \R$. Moreover, $f$ is quasiconcave on $C$ if and only if $C\cap S_{> \omega}^f$ is convex $\forall \, \omega$.

\begin{corollary}
\label{cor:regions-division}
Let $f$ be a homogeneous translatable polynomial of degree $d \ge 2$ with integer coefficients. In time and output size bounded by polynomial in $d$, the size of $R$, and the size of the coefficients of $f$, we can find a polynomial number of rational polyhedra $P_i$, $Q_j$ and rational lines $L_k$ such that $f$ is quasiconvex on $P_i$, quasiconcave on $Q_j$, and 
\begin{equation}
B\cap \Z^2 = \left(\bigcup_{i=1}^{\ell_1} P_i \cup \bigcup_{j=1}^{\ell_2} Q_j \cup \bigcup_{k=1}^{\ell_3} L_k\right) \cap \Z^2 .
\end{equation}
In particular, for a given $\omega \in \R$, this yields a division description for $S_{\le \omega}^f$ on $B$.
\end{corollary}
\begin{proof}
If $D_f \equiv 0$, then by Lemma~\ref{lem:D_f=0} $f$ has the form $f(\x) = (\ve c^T(\x - \t))^d$.  This function has one line of zeros and has the property that whenever $f > 0$, $f$ is convex and whenever $f < 0$, $f$ is concave.  By Theorem~\ref{thm:partition_Z2-translatable}, we hence divide $B\cap \Z^2$ into polyhedra where $f$ is quasiconvex or quasiconcave and some rational lines containing integer points.
If instead $D_f \not \equiv 0$, then by Corollary~\ref{cor:Df_homogeneous}, $D_f$ is homogeneous translatable of degree $3d - 4$. By Theorem~\ref{thm:partition_Z2-translatable}, we can cover $B \cap \Z^2$ with polyhedra where $D_f<0$ or $D_f >0$ and just lines.
Applying Lemma~\ref{lem:determine-quasiconvex}, shows that $f$ is quasiconvex or quasiconcave in these polyhedra.
By the definition of quasiconvexity, this yields a division description for $S_{\le \omega}^f$ for all $\omega$.
\end{proof}

\label{sec:homogeneous-main}

\begin{proof}[Proof of Theorem~\ref{thm:min_homogeneous_bounded}]
If $d \le 1$, the problem is a particular case of integer linear programming, which is polynomially solvable in fixed dimension \cite{Scarf1981,Scarf1981a,lenstra_integer_1983}. Assume that $d \ge 2$.
By Corollary~\ref{cor:regions-division}, we can construct a polynomial number of rational polyhedra with polynomially bounded size where $f$ is quasiconvex or quasiconcave, and linear spaces, that cover $P \cap \Z^2$.  This description yields a division description for any $\omega$.
We can then solve our problem in polynomial time using Theorem~\ref{thm:dd-min}. \end{proof}

The proof of Theorem~\ref{thm:min_homogeneous_bounded} is more general than is needed here.  In fact, the same proof
shows that we can also minimize a polynomial $f(x_1, x_2)$ whenever $D_f$ is homogeneous translatable and is not identically zero.  Minimizing general quadratics in two variables can then be done in this manner, because either $D_f$ has those properties, or we can compute a rational constant $c_0$ such that $f(\x) + c_0$ is homogeneous translatable.

\section{Cubic Polynomials and Unbounded Polyhedra}
\label{sec:cubic-unbounded}

In this section, we prove Theorem~\ref{thm:min_cubic_unbounded}, i.e., we show how to solve Problem~\eqref{eq:main} in two variables when $f$ is cubic and $P$ is allowed to be unbounded.  The behavior of the objective function on feasible directions of unboundedness is mostly determined by the degree three homogeneous part of $f$.  Hence we will study degree three homogeneous polynomials before proceeding with the proof of Theorem~\ref{thm:min_cubic_unbounded}

A homogeneous polynomial of degree three factors (over the reals) either into a linear function and an irreducible quadratic polynomial, or three linear functions.  We distinguish between the cases where the linear functions are distinct, have multiplicity two, or have multiplicity three, in part, by analyzing the  \emph{discriminant} $\Delta_d$ of certain degree $d$ polynomials.
In particular, if $p(x) = \sum_{i=0}^d c_i x^i$, with $c_d \ne 0$, then $\Delta_2 = c_1^2 - 4 c_0^{} c_2^{}$ and $\Delta_3 = c_1^2 c_2^2 - 4 c_1^3 c_3^{} - 4 c_0^{} c_2^3 - 27 c_0^2 c_3^2 + 18 c_0^{} c_1^{} c_2^{} c_3^{}$.

\begin{lemma}[repeated lines]
\label{lem:repeated-lines}
Let $h(x,y) = c_0 x^3 + c_1 x^2 y + c_2 x y^2 + c_3 y^3$ be a homogeneous polynomial of degree 3 with integer coefficients, that is, $c_i\in \Z$ for $i=0,1,2,3$.  Suppose $h \not \equiv 0$ and there exist $a_1,b_1,a_2,b_2 \in \R$ such that 
\begin{equation}
\label{eq:homogeneous-two-lines}
h(x,y) = (a_1 x + b_1y)^2 (a_2x + b_2y).
\end{equation}
Then, in polynomial time we can determine $a_i', b_i' \in \Q$, for $i = 1, 2$,  such that $h(x, y) = d \, (a_1' x + b_1' y)^2 (a_2' x + b_2' y)$ for some $d \in \R$.
\end{lemma}
\begin{proof}

First, suppose that $c_0 = c_3 = 0$.  Then 
$h(x, y) = x y (c_1 x + c_2 y)$, which implies $c_1 = 0$ or $c_2 = 0$ by equation~\eqref{eq:homogeneous-two-lines}. If $c_1 = 0$, then we can take $a_1' = 0$, $b_1' = 1$, $a_2' = c_2$ and $b_2' = 0$. The case $c_2 = 0$ follows by switching $x$ and $y$.  Henceforth, we consider the case where either $c_1\neq 0$ or $c_3 \neq 0$.  Since these cases are symmetric in switching $x$ and $y$, we assume without loss of generality that $c_0 \neq 0$.

Since $c_0 \ne 0$, by equation~\eqref{eq:homogeneous-two-lines} we have that $a_i \ne 0$ for $i = 1, 2$. Consider the rational cubic polynomial $h(x, 1) = c_0 x^3 + c_1 x^2 + c_2 x + c_3$. By equation~\eqref{eq:homogeneous-two-lines} it has a repeated root, so its discriminant $\Delta_3$ is equal to zero (see, for example, \cite{irving_integers_2004}).
Let $p = - \frac{c_1^2}{3 c_0^2} + \frac{c_2}{c_0}$ and $q = \frac{2 c_1^3}{27 c_0^3} - \frac{c_1 c_2}{3 c_0^2} + \frac{c_3}{c_0}$. Since $\Delta_3 = 0$, by \cite{irving_beyond_2013}
the roots of $h(x, 1)$ are of the form $r - \frac{c_1}{3 c_0}$, where either $r = 0$, $r = \pm \sqrt{-p/3}$ or $r = \pm 2 \sqrt{-p/3}$. Since $\Delta_3 = c_0^4(- 4 p^3 - 27 q^2) = 0$, we have that if $p \ne 0$, then $\sqrt{-p/3} = \frac{3 |q|}{2 |p|}$, which is rational. Thus all roots are rational and can be computed explicitly in terms of the coefficients of $h$.

Finally, if $r_1$ is the double root of $h(x, 1)$ and $r_2$ is the single root, then $a'_1 = 1$, $b'_1 = - r_1$, $a'_2 = 1$ and $b'_2 = - r_2$ are rational and $h(x, y) = d \, (a_1' x + b_1' y)^2 (a_2' x + b_2' y)$ for $d = a_1^2 a_2$.
\end{proof}

\begin{figure}
\centering
\begin{minipage}[b]{0.2\textwidth}
\centering
\includegraphics[width=\textwidth,trim={7mm 0 0 0}]{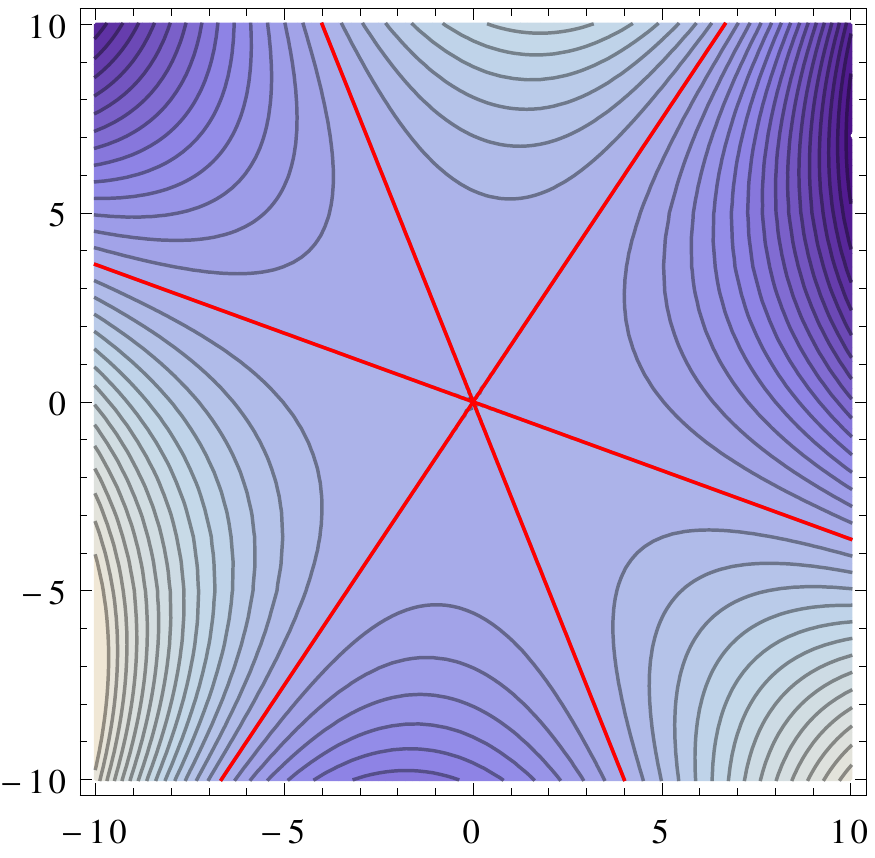}
\end{minipage}
\quad
\begin{minipage}[b]{0.2\textwidth}
\centering
\includegraphics[width=\textwidth,trim={7mm 0 0 0}]{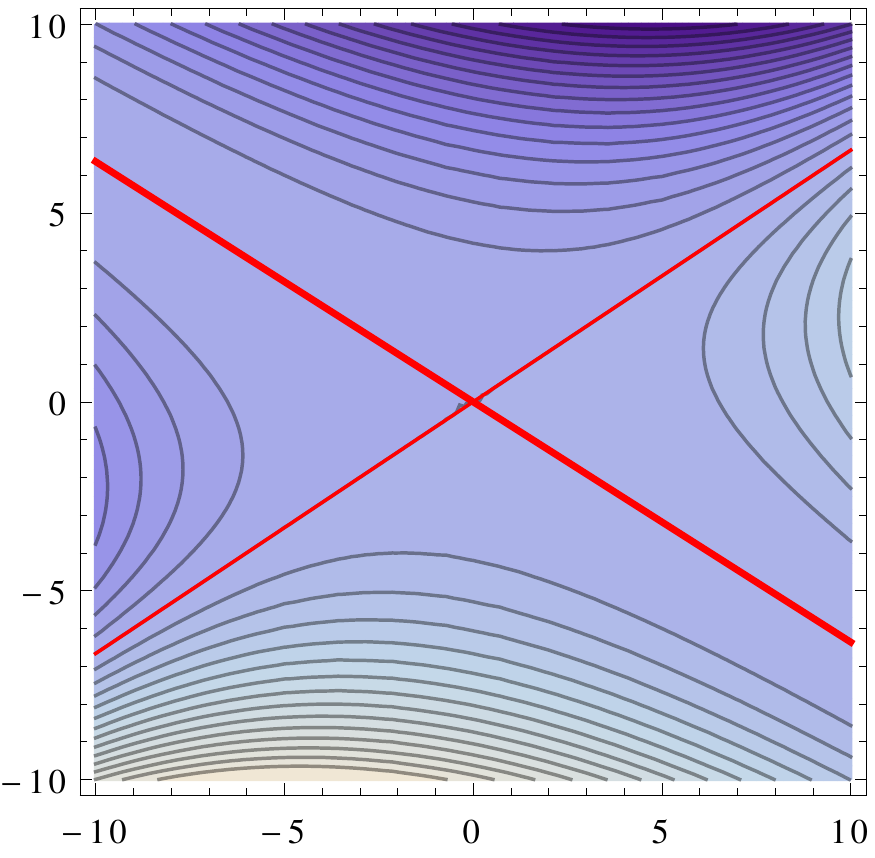}
\end{minipage}
\quad
\begin{minipage}[b]{0.2\textwidth}
\centering
\includegraphics[width=\textwidth,trim={7mm 0 0 0}]{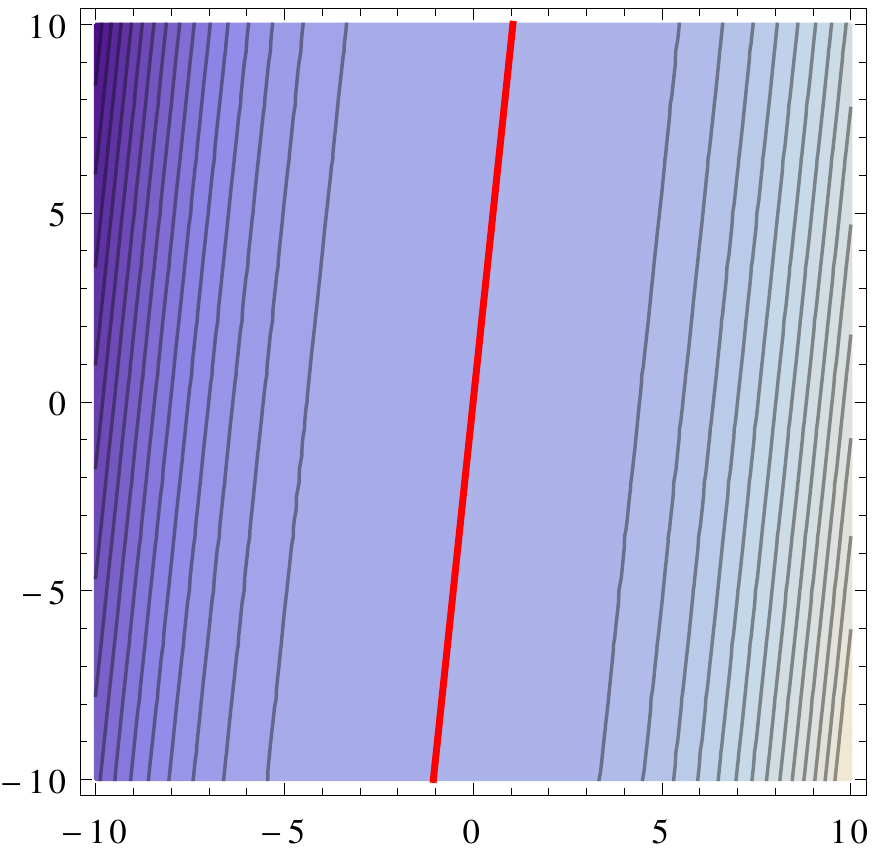}
\end{minipage}
\quad
\begin{minipage}[b]{0.2\textwidth}
\centering
\includegraphics[width=\textwidth,trim={7mm 0 0 0}]{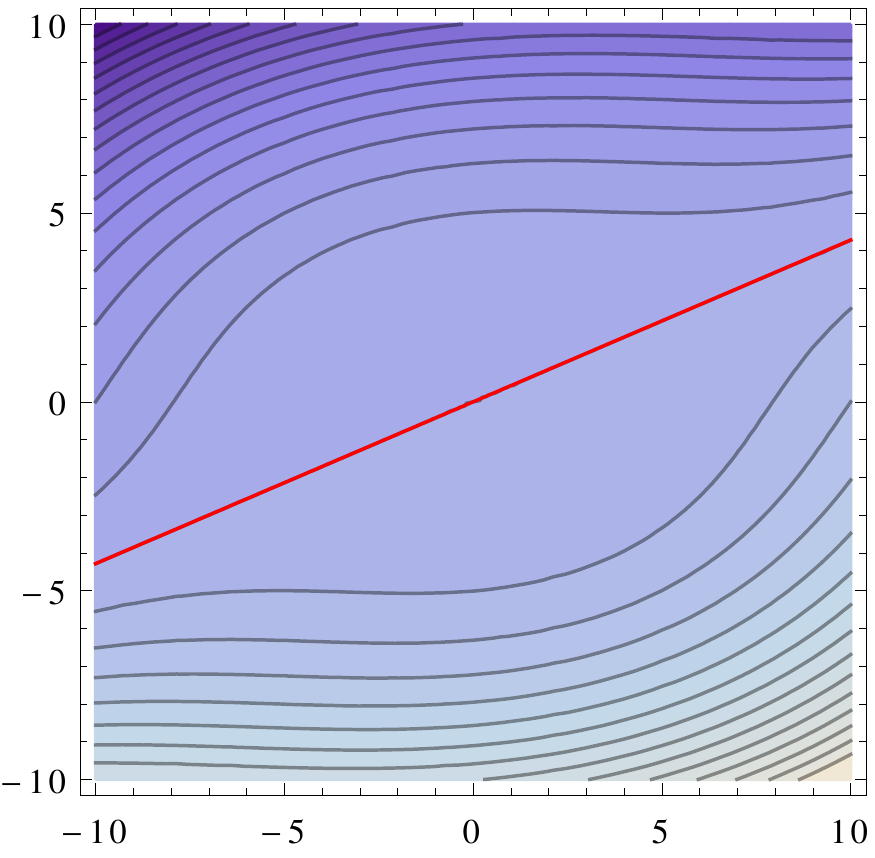}
\end{minipage}
\quad
\begin{minipage}[b]{0.2\textwidth}
\centering
\begin{tikzpicture}[scale= .32]
    \draw(-5,-5) -- (5,-5) -- (5,5) -- (-5,5) -- (-5,-5);
    \draw[red, thick] (-5,3) -- (5,-3);
    \draw[red, thick](-3,-5) -- (3,5);
    \draw[red, thick] (-2,5) -- (2,-5);
    \node at (-3,3) {$+$};
    \node at (0,3) {$-$};
    \node at (3,1) {$+$};
    \node at (3,-3) {$-$};
    \node at (0,-3) {$+$};
    \node at (-3,-1) {$-$};
    \path[fill=green, draw = black, opacity = 0.5] (5,-2) -- (0,0) -- (3,5)--(5,5);
    \node at (3.3,2) {{\small{$\rec(P)$}}};
\end{tikzpicture}
\\(i)
\end{minipage}
\quad
\begin{minipage}[b]{0.2\textwidth}
\centering
\begin{tikzpicture}[scale= .32]
    \draw(-5,-5) -- (5,-5) -- (5,5) -- (-5,5) -- (-5,-5);
    \draw[red, ultra thick](-5,3) -- (5,-3);
    \draw[red, thick](-5,-3) -- (5,3);
    \node at (3,0) {$-$};
    \node at (-3,0) {$+$};
    \node at (0,-3) {$-$};
    \node at (0,3) {$+$};
    \path[fill=green, draw = black, opacity = 0.5] (-5,-2) -- (0,0) -- (3,5) -- (-5,5);
    \node at (-3,4) {{\small{$\rec(P)$}}};
\end{tikzpicture}
\\(ii)
\end{minipage}
\quad
\begin{minipage}[b]{0.2\textwidth}
\centering
\begin{tikzpicture}[scale= .32]
    \draw(-5,-5) -- (5,-5) -- (5,5) -- (-5,5) -- (-5,-5);
    \draw[red, ultra thick](-1,-5) -- (1,5);
    \node at (3,-1) {$-$};
    \node at (-3,1) {$+$};
    \path[fill=green, draw = black, opacity = 0.5] (-5,-3) -- (0,0) -- (1,5) -- (-5,5);
    \node at (-3,4) {{\small{$\rec(P)$}}};
\end{tikzpicture}
\\(iii)
\end{minipage}
\quad
\begin{minipage}[b]{0.2\textwidth}
\centering
\begin{tikzpicture}[scale= .32]
    \draw(-5,-5) -- (5,-5) -- (5,5) -- (-5,5) -- (-5,-5);
    \draw[red, thick](-5,-2) -- (5,2);
    \node at (-1,3) {$+$};
    \node at (1,-3) {$-$};
    \path[fill=green, draw = black, opacity = 0.5] (-5,3) -- (0,0) -- (5,2) -- (5,5) -- (-5,5);
    \node at (2,4) {{\small{$\rec(P)$}}};
\end{tikzpicture}
\\(iv)
\end{minipage}

\caption{
The top four plots are contour plots for examples of the four classes of cubic homogeneous polynomials as described in the proof of Lemma~\ref{lem:determine-case}. The bottom four plots depict the sign of the polynomial. The red lines are the zero level set. The polynomial is zero on (i) three distinct lines, (ii) one single and one repeated line, (iii) a triple line or (iv) a single line. The thicker red line in (ii) and (iii) is the zero line with multiplicity 2 and 3 respectively. If the function is zero on a single line only as in (iv), then it is the product of a line and a quadratic polynomial that is irreducible over the reals.
Examples of $\rec(P)$ as discussed in Case~3 of the proof of Theorem~\ref{thm:min_cubic_unbounded} are depicted in green.
}
\label{fig:cubic-homogeneous-classes}
\end{figure}
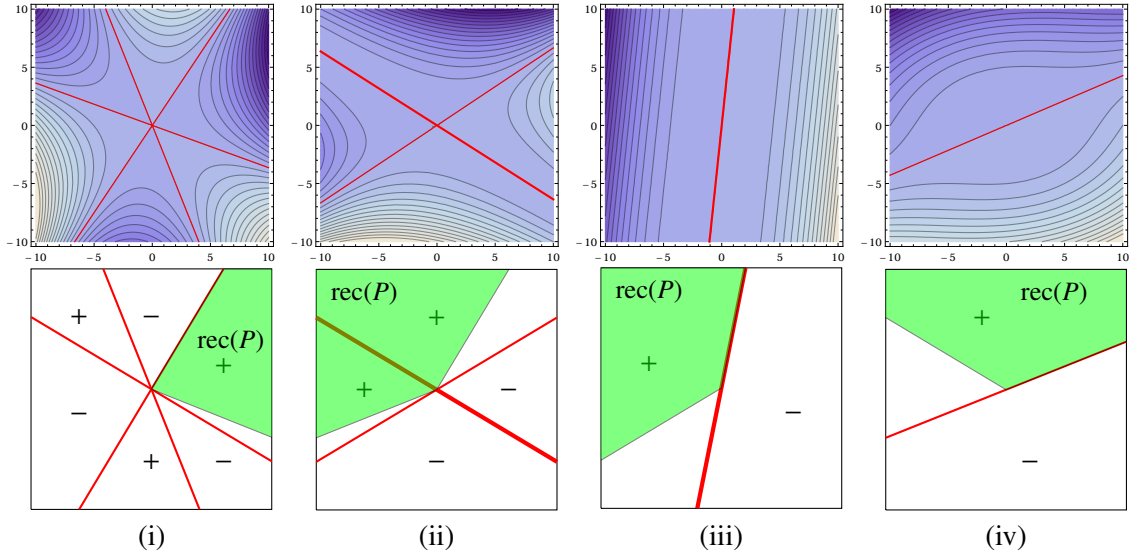

If $h$ is a cubic homogeneous bivariate polynomial, then a linear factor always factors out over the reals.  Thus, there exist real numbers $a_i, b_i$ for $i=1,2,3$, with $(a_i, b_i)$ and $(a_j, b_j)$ linearly independent if $i \neq j$, a quadratic irreducible polynomial $q(x,y)$, and $d \in \R$ such that it is of one of the following types:
\begin{enumerate}[\hspace{1cm}Type (i)]
\item $h(x, y) = (a_1 x + b_1 y)\,(a_2 x + b_2 y)\,(a_3 x + b_3 y)$,
\label{it:L1-L2-L3}
\item $h(x, y) = (a_1 x + b_1 y)^2\,(a_2 x + b_2 y)$,
\label{it:L1-L2}
\item $h(x, y) = d \, (a_1 x + b_1 y)^3$,
\label{it:L1}
\item $h(x,y) = (a_1 x + b_1 y) \, q(x, y)$.
\label{it:L-IR}
\end{enumerate}
An example of each type is shown in Figure~\ref{fig:cubic-homogeneous-classes}.

\begin{lemma}
\label{lem:determine-case}
Let $h$ be a cubic homogeneous bivariate polynomial. In polynomial time, we can determine which of the four Types \eqref{it:L1-L2-L3}-\eqref{it:L-IR} the polynomial $h$ belongs to.  Furthermore, if it is of Type~\eqref{it:L1-L2} or Type~\eqref{it:L1}, in polynomial time, we can compute rational $a_i, b_i$, $i=1,2$, satisfying the equation.
\end{lemma}
\begin{proof}
We show how to determine which of the Types~\eqref{it:L1-L2-L3}-\eqref{it:L-IR} the polynomial belongs to.
Suppose first that the coefficient of $y^3$ is zero. Then $x = 0$ is one of the zero lines of the polynomials and we can factor out $x$, leaving us with a homogeneous polynomial of degree 2. If $x$ factors out again, then depending on whether it factors out a third time or not, we are either in Type \eqref{it:L1-L2} or \eqref{it:L1}, because either $h(x, y) = a_1^2 x^2 (a_2 x + b_2 y)$, $b_2 \ne 0$, or $h(x, y) = a_1^3 x^3$.
If $x$ does not factor out a second time, then we look at the discriminant $\Delta_2$ of the degree 2 polynomial where we set $x = 1$. It is well known that the sign of $\Delta_2$ determines whether $h(1, y)$ has $2$, $1$ or $0$ distinct real roots (see, for example, \cite{irving_integers_2004}).
Remember that since $h$ is homogeneous, $h(1, \bar y) = 0$ if and only if $h(\lambda, \lambda \bar y) = 0$ $\forall \, \lambda \ge 0$. Thus, if $\Delta_2 > 0$ we are in Type~\eqref{it:L1-L2-L3}, if $\Delta_2 = 0$ in Type~\eqref{it:L1-L2}, and if $\Delta_2 < 0$ in Type~\eqref{it:L-IR}.

Let us now look at the case where $x = 0$ is not a zero line of the polynomial. Set $x = 1$ and consider the discriminant $\Delta_3$ of the resulting polynomial. It is well known that the sign of $\Delta_3$ determines whether $h(1, y)$ has three distinct real roots, one real root and two complex conjugate roots, or if all roots are real and at least two of them coincide (see, for example, \cite{irving_integers_2004}).
Thus, if $\Delta_3 > 0$, we are in Type \eqref{it:L1-L2-L3}, and if $\Delta_3 < 0$, we are in Type \eqref{it:L-IR}. Finally, if $\Delta_3 = 0$ we are either in Type~\eqref{it:L1-L2} or \eqref{it:L1}. To distinguish between those two, note that it is possible to establish the multiplicity of a root of a univariate polynomial by checking whether it is also a root of its derivatives. Thus compute the root of the second derivative. Note that it is a rational expression in terms of the coefficients of the original polynomial, so it is rational. If it is also a root of the first derivative and the original polynomial (with $x = 1$), then it is a triple root, so we are in Type~\eqref{it:L1}. Otherwise, we are in Type~\eqref{it:L1-L2}.
Hence, in polynomial time, we can determine which case we are in.

Finally, by Lemma~\ref{lem:repeated-lines}, if $h$ is of Type~\eqref{it:L1-L2} or \eqref{it:L1}, then we can compute rational $a_i, b_i$ for $i=1,2$. In fact, if $h$ is of Type~\eqref{it:L1-L2} this is straightforward, whereas if it is of Type~\eqref{it:L1}, we can rewrite $h$ as $h(x, y) = (a_1 x + b_1 y)^2\,(d a_1 x + d b_1 y)$.
\end{proof}

We will also need the following lemma about lower bounding a polynomial that is positive on a compact set.
\begin{lemma}
\label{lem:lower-bound}
Let $f\colon [0,1] \to \R$ be a polynomial of maximum degree at most $3$ with rational coefficients.  Suppose $f(x) > 0$ for all $x \in [0,1]$.  Then there exists a lower bound $m$ of polynomial size in the size of the coefficients of $f$ such that $f(x) > m > 0$ for all $x \in [0,1]$.
\end{lemma}
\begin{proof}
We will bound the minimum value of $f$ on $[0,1]$.  Since $f$ is a polynomial, it attains its minimum value at a critical value or at one of the endpoints of the interval.  Clearly, $f(0)$ and $f(1)$ have a polynomial size in the size of the coefficients of $f$, so we only need to bound the function at any critical values.

Consider 
$$
f(x) = ax^3 + b x^2 + cx + d, \quad \text{and } \quad f'(x) = 3a x^2 + 2 b x + c = 0.
$$
If $a = 0$, then the only critical point is at $x = -c/(2b)$.  Clearly $f(-c/(2b))$ also has polynomial size.

Otherwise, $a \neq 0$.  Then, from the quadratic formula, it follows that  

$$
x_\pm = \frac{-b \pm \sqrt{b^2 - 3ac}}{3a}, \quad \text{and} \quad 
f(x_\pm) = \frac{27 a^2 d-9 a b c+2 b^3 \pm \left(6 a c - 2 b^2\right) \sqrt{b^2-3 a c} }{27 a^2}.
$$
Thus, set $p = \frac{27 a^2 d-9 a b c+2 b^3}{27 a^2}$, $q = \pm \frac{6 a c - 2 b^2}{27 a^2}$, and $t= b^2-3 a c$, and let $f^* = p + q \sqrt{t}$.
By rearranging, we arrive at 
$$
(f^* - p)^2 - q^2 t = 0.
$$
Since we know that $f^* > 0$, by Theorem~\ref{thm:polynomial_root_bounds} its smallest positive root is at least as large as some $m > 0$, where $m$ is of polynomial size in the original coefficients.
\end{proof}

We now prove Theorem~\ref{thm:min_cubic_unbounded}.

\begin{proof}[Proof of Theorem~\ref{thm:min_cubic_unbounded}]
We will either show that Problem~\eqref{eq:main} is unbounded and exhibit a feasible point $\bar \x$ and an integral ray $\bar \r$, each of polynomial size, such that the objective function tends to  $-\infty$ along $\bar \x + \lambda \bar \r$ as $\lambda \to \infty$, or we will exhibit a polynomial size bound $R$ such that the optimal solution $\x^*$ of Problem~\eqref{eq:main} is contained in $P \cap [-R,R]^2$, and hence the solution can be found in polynomial time using Theorem~\ref{thm:cubic}.

We assume that $P_I$ is unbounded, because otherwise Problem~\eqref{eq:main} can be solved using Theorem~\ref{thm:cubic}.  If the degree of $f$ is not greater than two, then Problem~\eqref{eq:main} can be solved by \cite{delpia_integer_2013}.  Thus we will assume that the degree of $f$ is exactly three.
We will also assume that $\rec(P)$ is a pointed cone, where $\rec(P)$ denotes the recession cone of $P$.  If not, then we can divide $P$ into four polyhedra where the recession cone is pointed, for instance, by restricting to the four standard orthants of $\R^2$, and solving on each polyhedron separately.  
Since $\rec(P)$ is pointed, using linear programming techniques, we can compute rational rays $\r^1, \r^2$ such that $\rec(P) = \cone \{\r^1, \r^2\} = \{ \lambda_1 \r^1 + \lambda_2 \r^2 : \lambda_1, \lambda_2 \geq 0\}$.   Without loss of generality, we assume that $\r^1, \r^2 \in \Z^2$ are integral, as this can be obtained by scaling.

 We will be focused on the behavior of $f(\x + \lambda \r)$ as $\lambda$ varies, where $\x \in P \cap \Z^2$ and $\r \in \rec(P)$.  Since $\rec(P) = \cone \{\r^1, \r^2\}$, we only need to restrict our attention to $\r \in \conv\{\r^1, \r^2\}$.  Since $P$ is pointed, $\ve 0 \notin \conv\{\r^1, \r^2\}$.  In large part, the behavior of $f(\x + \lambda \r)$ is determined by $h$, where $h$ is the non-trivial degree three homogeneous polynomial such that $f - h$ is of degree two. 
We decompose $f(\x + \lambda \r)$ by parametrizing with $\lambda$ in the following way:
\begin{equation}
\label{eq:f-decompose}
f(\x + \lambda \r) = h(\r) \lambda^3 + g_2(\x, \r) \lambda^2 + g_1(\x, \r) \lambda + f(\x).
\end{equation}

We consider cases based on the sign of $h$ on $\conv\{\r^1, \r^2\}$.  In particular, let $h^* = \min\{ h(\r) : \r \in \conv\{\r^1, \r^2\}\}$.  We can determine the sign of $h^*$ by considering the univariate polynomial $\bar h(s) := h(s \r^1 + (1-s) \r^2)$ on the interval $s \in [0,1]$ and applying Lemma~\ref{lem:numerical} part~\eqref{part3}  and testing the points based on the location of the zeros.  Note that the sign of $h^*$ can be determined without actually computing $h^*$. This is important, since the value $h^*$ could be irrational.

\textbf{\underline{Case 1:} Suppose $h^* < 0$}.

Then $h(\r) < 0$ for some $\r \in \conv\{\r^1, \r^2\}$.   We begin by computing a rational ray $\bar \r \in \conv\{\r^1, \r^2\}$ with $h(\bar \r) < 0$.    If $h(\r^1) < 0$ or $h(\r^2) < 0$, then we are done.  Otherwise, by Rolle's theorem, $\bar h(s) = 0$ somewhere on $[0,1]$.  By numerically approximating the zeros of $h$ with Lemma~\ref{lem:numerical} part~\eqref{part3}, we obtain separated upper and lower bounds, $\tilde \alpha_i^+$ and $\tilde \alpha_i^-$, on the zeros of $h$.  Note that any prescribed tolerance $\epsilon$ works here since we only want to separate the zeros.  Once the zeros are separated, one of these approximations must attain a negative value, call it $\tilde s$.  Then set $\bar \r = \tilde s \r^1 + (1-\tilde s) \r^2$.

Using Lenstra's algorithm, find a point $\bar\x \in P\cap \Z^2$ of polynomially bounded size. There exist infinitely many points $\bar\x + \lambda \bar\r \in P \cap \Z^2$ for $\lambda \geq 0$.   By~\eqref{eq:f-decompose}, $f(\bar \x + \lambda \bar \r)$ is a cubic polynomial in $\lambda$ with a negative leading coefficient.  Hence $f(\x + \lambda \r) \to - \infty$ as $\lambda \to \infty$, i.e., Problem~\eqref{eq:main} is unbounded from $\bar \x$ along the ray $\bar \r$.

\textbf{\underline{Case 2:} Suppose $h^* > 0$}.

Since $h^* > 0$ and $\r^1, \r^2$ are rational of polynomial size, and $\bar h$ is a polynomial with coefficients of polynomial size, we can also pick an $0 < h_0 \leq h^* $ of polynomial size using Lemma~\ref{lem:lower-bound}.

Using the Minkowski-Weyl theorem and linear programming, we can decompose $P$ into $P = Q + \rec(P)$, where $Q$ is a polytope and hence bounded.   Let $R_Q$ be of polynomial size such that $Q \subseteq [-R_Q, R_Q]^2$.

Again, using Lenstra's algorithm, determine any $\bar\x \in P \cap \Z^2$ of polynomial size.
We now will determine a bound on $\lambda$ such that $f(\q + \lambda \r) \geq f(\bar\x)$ for all $\q \in Q$ and $\r \in \conv\{\r^1, \r^2\}$.  

From~\eqref{eq:f-decompose}, we have
$
f(\q + \lambda \r) - f(\bar\x)= h(\r) \lambda^3 + g_2(\q, \r) \lambda^{2} + g_1(\q, \r) \lambda + f(\q) - f(\bar\x).
$
Note that $g_i$ are polynomials with coefficients of size bounded by a polynomial in the size of the coefficients of $f$.  Since $(\q, \r) \in Q \times \conv\{\r^1, \r^2\}$, there exists a uniform polynomial size upper bound on $|f(\q) - f(\bar\x)|, |g_i(\q,\r)|$, $i=1,2$, on $Q \times \conv\{\r^1, \r^2\}$, call this $U$. 
Then
$$
f(\x) - f(\bar\x) \geq h(\r) \lambda^3 - 3 U \lambda^{2} \geq h_0 \lambda^3 - 3 U \lambda^{2} = (h_0 \lambda - 3 U) \lambda^2 \geq 0,
$$
where first and last inequalities together hold whenever $\lambda \geq \max\{\frac{3 U}{h_0}, 1\}$.  Therefore, if $\x = \q + \lambda \r \in P \cap \Z^2$ with $\lambda \geq \frac{3U}{h_0}$, then $f(\x) \geq f(\bar\x)$.

Finally, let $R := R_Q + \max\left\{ 1, \frac{3 U}{h_0} \right\} \times \max\left\{\| \r^1 \|_2^2, \| \r^2\|_2^2\right\}$.

Therefore, the optimal solution to Problem~\eqref{eq:main} is contained in $[-R,R]^2$ and $R$ is of polynomial size.

\textbf{\underline{Case 3:} Suppose $h^* \geq 0$}.

If $h^* > 0$, then we are in Case 2.  Otherwise, there exists at least one ray $\r \in \conv\{\r^1, \r^2\}$ such that $h(\r) = 0$.
Now $\rec(P) \subseteq S^h_{\ge 0}$.
Recall that the homogeneous polynomial $h$ is of one of the four types (Types~\eqref{it:L1-L2-L3}-\eqref{it:L-IR}).  Notice that Type~\eqref{it:L1-L2} is the only type where $\intr(\rec(P)) \cap  S^h_{=0}$ can be non-empty (cf.\ Figure~\ref{fig:cubic-homogeneous-classes}).  Therefore, if $h$ is not of Type~\eqref{it:L1-L2}, then $\r$ must be an extreme ray of $\rec(P)$, i.e., $\r=\r^1$ or $\r = \r^2$.  Otherwise, when $h$ is of Type~\eqref{it:L1-L2}, the set $S^h_{=0}$ is a rational line.  By Lemma~\ref{lem:determine-case}, we can distinguish these cases and compute the rational line if necessary.  Hence, we can compute  all candidate rays $\r\in \conv\{\r^1, \r^2\}$ such that $h(\r) = 0$ in polynomial time.   

We now divide the problem such that we must consider at most one of these recession rays and such that it is an extreme ray of the divided problem. This can be done, for instance, by averaging the candidate rays,  then restricting recession cones described by neighboring pairs of rays. We show how to solve just one such problem, as the others can be solved in an identical manner.  Thus, for the remainder of the proof, we redefine $\r^1, \r^2$ such that $\rec(P) = \cone \{ \r^1, \r^2 \}$ with $h(\r^1) = 0$, $h(\r) > 0$ for all $\r \in \conv\{\r^1, \r^2\} \setminus \{\r^1\}$.  Without loss of generality, $\r^1, \r^2 \in \Z^2$.

Finally, we make one more decomposition.   Let $Q$ be the convex hull of the vertices of $P_I$, in particular $P_I = Q + \rec(P)$, and $Q$ is bounded. Let $\hat \x$ be the vertex of $Q$ such that $P_I = P_1 \cup P_2$ where $P_1 = \hat \x + \rec(P)$ and $P_2 = Q + \cone\{\r^2\}$. 
Choose $(\r^1)^\perp$ as either $(-r^1_2, r^1_1)$ or $(r^1_2, -r^1_1)$ such that $\r^2 \cdot (\r^1)^\perp > 0$. Then the vertex $\hat \x$ is a solution to the minimization problem $\min\{ (\r^1)^\perp\cdot  \x :  \x \in P \cap \Z^2\}$.

Notice that $\rec(P_2) = \cone(\r^2)$ and $h(\r^2) > 0$.  Thus, we can derive a bound just as in Case 2, that is, we find polynomial size bounds $\bar \lambda_2$ and $R_2$ such that the optimal solution in $P_2$ is contained in $[-R_2, R_2]^2$ and also $f(\q + \lambda_2 \r^2) \geq f(\hat \x)$ for all $\lambda_2 \geq \bar \lambda_2$ and $\q \in Q$.

Henceforth, we only need to focus on the region $P_1 = \hat \x + \rec(P)$.

We will use equation~\eqref{eq:f-decompose} with $\r = \r^1$ fixed.  To make explicit that $g_1$, $g_2$ only depend on $\x$, we write $g_1^{\r^1}(\x) := g_1(\x,\r^1)$ and $g_2^{\r^1}(\x) := g_2(\x,\r^1)$.   Since we know $\r^1$, we can compute explicitly the coefficients in the polynomials $g_1^{\r^1}(\x)$ and $g_2^{\r^1}(\x)$.   Since $h(\r^1) = 0$, equation~\eqref{eq:f-decompose} reduces to
\begin{equation}
f(\hat \x + \lambda \r^1 + \lambda_2 \r^2) = g_2^{\r^1}(\hat \x + \lambda_2 \r^2) \lambda^2 + g_1^{\r^1}(\hat \x+ \lambda_2 \r^2) \lambda + f(\hat \x + \lambda_2 \r^2).
\end{equation}

\textbf{Case 3a: Suppose $g_2^{\r^1}(\x) \not\equiv 0$.}
Now consider the equivalent rewritings of $f(\x + \mu \r^1 + \lambda \r^1)$ as polynomials in $\lambda$, 
\begin{align*}
f((\x + \mu \r^1) + \lambda \r^1) &= g_2^{\r^1}(\x + \mu \r^1) \lambda^2 + g_1^{\r^1}(\x + \mu \r^1) \lambda + f(\x+ \mu \r^1),\\
f(\x + (\mu + \lambda) \r^1) &= g_2^{\r^1}(\x) (\mu + \lambda)^2 + g_1^{\r^1}(\x) (\mu +\lambda) + f(\x).
\end{align*}
Considering these as polynomials in $\lambda$, the highest order coefficients, i.e., the coefficients of $\lambda^2$, must coincide.  Hence $g_2^{\r^1}(\x + \mu \r^1) = g_2^{\r^1}(\x)$, so $g_2$ is invariant with respect to changes in the $\r^1$ direction.  

Therefore, we compute 
\begin{equation}
\label{eq:min-g2}
\begin{split}
g_2^* &:= \min \{ g_2^{\r^1}(\x) :  \x \in P_1\cap \Z^2 \}
= \min \{ g_2^{\r^1}(\hat \x + \lambda_2 \r^2) :  \lambda_1 \r^1 + \lambda_2 \r^2 \in \Z^2, \lambda_1, \lambda_2 \geq 0 \} 
\\
&\phantom{:}= \min \{ g_2^{\r^1}(\hat \x + \lambda_2 \r^2) :  \lambda_2  \in \tfrac{p}{q} \Z_+ \},
\end{split}
\end{equation}
where $p = \gcd \{ r^1_1, r^1_2 \}$, and $q = (\r^1)^\perp \cdot \r^2$.  The last inequality holds since $\exists \x \in \Z^2$ such that $(\r^1)^\perp \cdot \x = \z \in \Z$ if and only if $\z \in p \Z$ and moreover $(\r^1)^\perp \cdot \x  = (\r^1)^\perp \cdot(\hat \x + \lambda_1 \r^1 + \lambda_2 \r^2) = (\r^1)^\perp \cdot(\hat \x + \lambda_2 \r^2)$.  Therefore, $\lambda_2 (\r^1)^\perp \cdot \r^2 \in p \Z$, i.e., $\lambda_2 \in \tfrac{p}{q} \Z$.  
This last problem is a one dimensional integer polynomial optimization problem that can be solved by Theorem~\ref{thm:dimension1}.

We now do a case analysis on the sign of $g_2^*$ provided that $g_2^{\r^1}(\x) \not\equiv 0$.

\textbf{Case 3a1: Suppose $g_2^* < 0$.}
Then let $\lambda_2^*$ be a minimizer to the minimization problem \eqref{eq:min-g2}.  Let $\bar \x \in \Z^2$ be such that $\bar \x = \hat \x + \lambda_2^*\r^2 + \lambda_1 \r^1 \in \Z^2$ for some $\lambda_1 \geq 0$, which can be found using a linear integer program.
Since $f(\bar\x + \lambda \r^1)$ is a quadratic polynomial in $\lambda$ with a negative leading coefficient and $f(\bar \x + \lambda \r^1) \to -\infty$ as $\lambda \to \infty$, the problem is unbounded from the point $\bar \x$ along the ray $\r^1$.

\textbf{Case 3a2: Suppose $g_2^* > 0$.}
Since all inputs are integral, $g_2^{\r^1}(\hat \x) \in \Z$ and $g_2^* >0$, we have $g_2^{\r^1}(\hat \x) \geq 1$.
Since $g_2^{\r^1}(\x)$ is linear in $\x$, we have for $\ve \epsilon \in \R^2$, with $\| \ve \epsilon \|_\infty \le 1$
$$
| g_2^{\r^1}(\hat \x) - g_2(\hat \x, \r^1 + \ve \epsilon)| \leq 45 \times M \times \max \{1, \|\hat \x\|_\infty\} \times \|\r^1\|_\infty \times \|\ve \epsilon\|_\infty.
$$
Recall that $M$ is the sum of the absolute values of the coefficients of $f$.
We choose $\ve \epsilon = (\r^2 - \r^1) \, \epsilon$ such that $0 < \epsilon < \frac{1}{\|\r^2 - \r^1\|_2}$, and such that $\epsilon < \frac{1}{2 \times 45 \times M \times \max \{1, \|\hat \x\|_\infty\} \times \|\r^1\|_\infty }$.  Let $\hat \r = \r^1 + (\r^2 - \r^1) \, \epsilon$.  It follows that $\hat \r$ is of polynomial size and for all $\r \in \conv\{\r^1, \hat \r\}$, we have
$$
g_2(\hat \x,\r) \geq g_2^{\r^1}(\hat \x) - |g_2^{\r^1}(\hat \x) - g_2(\hat \x,\r)|  \geq 1 - \frac{1}{2} = \frac{1}{2}.
$$

We now decompose $P_1$ into pieces $P_{11} = \hat \x + \cone \{ \r^1, \hat \r \}$ and $P_{12} = \hat \x + \cone \{ \hat \r, \r^2 \}$.

On $P_{12}$,  a polynomial size bound $R_{12}$ is given by the analysis in Case~2 since $h(\r) > 0$ for all $\r \in \conv\{\hat \r, \r^2\}$.   

On $P_{11}$, consider any $\hat \x \in P_{11}$. Similar to Case~2, we find a $\bar \lambda$ of polynomial size such that $f(\hat \x + \lambda \r) \geq f(\hat \x)$ for all $\lambda \geq \bar \lambda$ and $\r \in  [\r^1, \hat \r]$.  We determine $\bar \lambda$ this time using the fact that $g_2 > 0$.   Using equation~\eqref{eq:f-decompose} and the fact that $h(\r) \geq 0$, we have
$$
f(\hat \x + \lambda \r) - f(\hat \x)= h(\r) \lambda^3 + g_2(\hat \x, \r) \lambda^2 + g_1(\hat \x, \r) \lambda + f(\hat\x) - f(\hat\x) \geq \lambda \big(g_2(\hat \x, \r) \lambda + g_1(\hat \x, \r) \big) \geq 0.
$$
This last inequality holds when $g_2(\hat \x, \r) \lambda \geq |g_1(\hat \x, \r)|$.  We can write down $g_i(\hat \x, \r)$, $i=1,2$, explicitly as polynomials of $\r$ with coefficients of size bounded by a polynomial in the size of the coefficients of $f$.  Since $\r \in  \conv\{\r^1, \hat \r\}$, which is a compact set, there exists a uniform polynomial size upper bound on $ |g_1(\hat \x,\r)|$, call this $U$.   
Hence, we can choose $\bar \lambda = U$.

Finally, let $R_{11} := \|\hat \x\|_\infty + \max\left\{1, U \right\} \times \max\left\{\| \r^1 \|^2_2, \| \hat \r\|^2_2\right\}$.
Therefore, the optimal solution in $P_1$ to Problem~\eqref{eq:main} is contained in $[-R,R]^2$ for $R = \max\left\{R_{11}, R_{12}\right\}$, which is of polynomial size since $R_{11}, R_{12}$ are of polynomial size.

\textbf{Case 3a3: Suppose $g_2^* = 0$.}
Since $g_2^{\r^1}(\x) \not\equiv 0$, there are only polynomially many points where $g_2^{\r^1}(\hat \x + \lambda_2 \r^2) = 0$ for $\lambda_2 \in \tfrac{1}{q} \Z_+$.  After repeated application of Theorem~\ref{thm:dimension1}, we can find all such points, and call them $\lambda_{2,1}, \dots, \lambda_{2,m}$.  In fact, we can show that $g_2^{\r^1}(\hat \x + \lambda_2 \r^2)$ is linear in terms of $\lambda_2$, but we use the more general technique here as it will be repeated in Case 3b3.  

Each subproblem $\min\{f(\hat \x + \lambda_1 \r^1 + \lambda_{2,i} \r^2): \hat \x + \lambda_1 \r^1 + \lambda_{2,i} \r^2 \in P_1 \cap \Z^2\}$ can be converted into a univariate subproblem and solved with Theorem~\ref{thm:dimension1}.

The remaining integer points in the feasible region are contained in the polyhedra $P_1^i = P_1 \cap \{ \hat \x + \x : \lambda_{2,i}  (\r^1)^\perp \cdot \r^2 \leq (\r^1)^\perp \cdot \x \leq \lambda_{2,i+1}  (\r^1)^\perp \cdot \r^2\}$ for $i=0, \dots, m+1$, where we define $\lambda_{2,0} = 0$ and $\lambda_{2,m+1} = \infty$.

As before, decompose each $P_1^i$ as we did with $P$, into the polyhedra $Q_1^i + \cone\{\r^2\}$ and $\hat \x^i + \rec(P_1^i)$.  As before with $P_2$,  the optimal solution on $Q_1^i + \cone\{\r^2\}$ can be bounded using the techniques of Case 2.   
On each subproblem $\hat \x^i + \rec(P_1^i)$, we have $g_2^{\r^1}(\hat \x^i) > 0$; hence we can apply the techniques of Case 3a2.

\textbf{Case 3b: Suppose $g_2^{\r^1}(\x) \equiv 0$, but $g_1^{\r^1}(\x) \not \equiv 0$.}

Consider the equivalent rewritings of $f(\x + \mu \r^1 + \lambda \r^1)$ as polynomials in $\lambda$, 
\begin{align*}
f((\x + \mu \r^1) + \lambda \r^1) &= g_1^{\r^1}(\x + \mu \r^1) \lambda + f(\x+ \mu \r^1),\\
f(\x + (\mu + \lambda) \r^1) &=  g_1^{\r^1}(\x) (\mu +\lambda) + f(\x).
\end{align*}
Considering these as polynomials in $\lambda$, the coefficients of $\lambda$ must coincide.  Hence $g_1^{\r^1}(\x + \mu \r^1) = g_1^{\r^1}(\x)$, so we see that $g_1^{\r^1}$ is invariant with respect to changes in the $\r^1$ direction.  

Therefore, we compute 
\begin{equation}
\label{eq:min-g1}
\begin{split}
g_1^* &:= \min \{ g_1^{\r^1}(\x) :  \x \in P_1\cap \Z^2 \} = \min \{ g_1^{\r^1}(\hat \x + \lambda_2 \r^2) :  \lambda_1 \r^1 + \lambda_2 \r^2 \in \Z^2, \lambda_1, \lambda_2 \geq 0 \} 
\\
&\phantom{:}= \min \{ g_1^{\r^1}(\hat \x + \lambda_2 \r^2) :  \lambda_2  \in \tfrac{p}{q} \Z_+ \},
\end{split}
\end{equation}
where $p = \gcd \{ r^1_1, r^1_2 \}$, $q = (\r^1)^\perp \cdot \r^2$, and the last inequality holds just as in Case~3a.
This last problem is a one-dimensional integer polynomial optimization problem that can be solved by Theorem~\ref{thm:dimension1}.

We now do a case analysis on the sign of $g_1^*$ provided that $g_1^{\r^1}(\x) \not\equiv 0$.  

\textbf{Case 3b1: Suppose $g_1^* < 0$.}
Similar to Case 3a1, there is a point $\bar \x \in P \cap \Z^2$ such that $g_1^{\r^1}(\bar \x) = g_1^* < 0$.
Then $f(\bar \x + \lambda_1 \r^1) = g_1^{\r^1}(\bar \x) \lambda_1 + f(\bar \x) \to -\infty$ for $\lambda_1 \to \infty$, so the problem is unbounded from the point $\bar \x$ in the direction $\r^1$.

\textbf{Case 3b2: Suppose $g_1^* > 0$.}  Since the minimization problem for $g_1^*$ is discrete and the objective is at most quadratic in $\lambda_2$, we have that $g_1^* \geq \tfrac{1}{q^2}$.  

First, for $\lambda_1 \ge 0$ we have that 
\begin{align*}
f(\hat \x + \lambda_1 \r^1 + \lambda_2 \r^2) &= h(\r^2) \lambda_2^3 + g_2(\hat \x, \r^2) \lambda_2^2 + g_1(\hat \x, \r^2) \lambda_2 + g_1^{\r^1}(\hat \x + \lambda_2 \r^2) \lambda_1 + f(\hat \x)
\\
&\ge h(\r^2) \lambda_2^3 + g_2(\hat \x, \r^2) \lambda_2^2 + g_1(\hat \x, \r^2) \lambda_2 + f(\hat \x) .
\end{align*}
In particular, $f(\hat \x + \lambda_1 \r^1 + \lambda_2 \r^2) \ge f(\hat \x)$ if $h(\r^2) \lambda_2^2 + g_2(\hat \x, \r^2) \lambda_2 + g_1(\hat \x, \r^2) \ge 0$. By Theorem~\ref{thm:polynomial_root_bounds} and $h(\r^2) > 0$, we have that this is satisfied for $\lambda_2 \ge \bar \lambda_2 := 1+\frac{1}{|h(\r^2)|}\max\{|g_2(\hat \x, \r^2)|,|g_1(\hat \x, \r^2)|\}$, which is of polynomial size.
Thus for $\lambda_2 \ge  \bar \lambda_2$  and $\lambda_1 \geq 0$, we have that $f(\hat \x + \lambda_1 \r^1 + \lambda_2 \r^2) \ge f(\hat \x)$.

Next, let 
\begin{align*}
L &\leq \min\{ h(\r^2) \lambda_2^3 + g_2(\hat \x, \r^2) \lambda_2^2 + g_1(\hat \x, \r^2) \lambda_2 : \lambda_2 \geq 0\}
\\
&= \min\{ h(\r^2) \lambda_2^3 + g_2(\hat \x, \r^2) \lambda_2^2 + g_1(\hat \x, \r^2) \lambda_2 : \lambda_2 \in [0,\bar \lambda_2]\}.
\end{align*}
The last inequality holds because either the minimum is zero, in which case $\lambda_2 = 0$ is a minimizer, or the minimum is negative, in which case there is a minimizer in $[0, \bar \lambda_2]$ because all zeros lie in this interval.
Since $[0,\bar \lambda_2]$ is compact and polynomially bounded, we can find a $L$ of polynomial size.  Then 
\begin{align*}
f(\hat \x + \lambda_1 \r^1 + \lambda_2 \r^2) &= h(\r^2) \lambda_2^3 + g_2(\hat \x, \r^2) \lambda_2^2 + g_1(\hat \x, \r^2) \lambda_2 + g_1^{\r^1}(\hat \x + \lambda_2 \r^2) \lambda_1 + f(\hat \x)
\\
&\ge L + \frac{\lambda_1}{q^2} + f(\hat \x).
\end{align*} 
Thus, $f(\hat \x + \lambda_1 \r^1 + \lambda_2 \r^2) \geq f(\hat \x)$ when $\lambda_1 \geq \bar \lambda_1 := - L q^2$.
Set $R := \bar \lambda_1 \times \|\r^1\|^2_2 +  \bar \lambda_2 \times \|\r^2\|^2_2$.  Then the optimal solution on $P_1$ is bounded in $[-R,R]^2$.

\textbf{Case 3b3: Suppose $g_1^* = 0$.}
This is similar to Case 3a3.

\textbf{Case 3c: Suppose $g_2^{\r^1}(\x) \equiv 0$ and $g_1^{\r^1}(\x) \equiv 0$.}

In this case, $f(\x) = f(\x + \lambda \r^1)$ for all $\lambda$.  Let $p, q$ be as in Case~3.  Then 
$$
\min \{ f(\x) : \x \in P_1 \cap \Z^2 \} = \min \{ f(\hat \x + \lambda_2 \r^2) : \lambda_2 \in \tfrac{p}{q}\Z_+ \},
$$
which again can be solved using Theorem~\ref{thm:dimension1}.  
\end{proof}

\begin{APPENDICES}

\section{Additional Proofs}

In this section we give some proofs that we omit from the main part of the paper.

\begin{proof}[Proof of Lemma~\ref{lem:determine-quasiconvex}.]
We will just prove statement $(\ref{itm:Dfsmaller0})$, as the proof for $(\ref{itm:Dflarger0})$ is similar.
Note that $D_f^1(\x) = -f_1(\x)^2 \leq 0$, so we need to establish that it is strictly negative.  Clearly this restriction is symmetric in $x_1$ and $x_2$.  Hence, we only need either   $-f_1(\x)^2 < 0$ or $-f_2(\x)^2< 0$.  Suppose that $f_1(\x)^2 = f_2(\x)^2 = 0$, in particular, $\nabla f(\x) = 0$.  By definition of $D_f$, expanding the determinant along the first row or first column shows that $D_f(\x) = 0$, which is a contradiction.  Therefore, $\nabla f(\x) \neq \ve 0$.  Hence, by Theorem~3.4.13 in \cite{cambini_generalized_2009}, $f$ is quasiconvex on $S$, and by Theorem~2.2.12 in \cite{cambini_generalized_2009}, $f$ is quasiconvex on the closure of $S$.
\end{proof}

\begin{proof}[Proof of Lemma~\ref{lemma:Df-det}.]
Applying Euler's Theorem for homogeneous functions
to the gradient, we obtain the equation 
\begin{equation}
\label{eq:euler2}
\nabla h(\x) = \frac{1}{d-1} \nabla^2 h(\x) \cdot \x \ \text{ for all } \x \in \R^n.
\end{equation}
Consider the linear combination of columns of $H_f(\x)$ given by 
$$
\begin{bmatrix} \nabla h^T(\x)\\   \nabla^2 h(\x)
\end{bmatrix} \x
= \begin{bmatrix} \nabla h^T(\x)\cdot \x\\   \nabla^2 h(\x)\cdot \x
\end{bmatrix} 
= \begin{bmatrix}d \, h(\x) \\    (d-1) \nabla h(\x)
\end{bmatrix},
$$
where the last equation comes from applying Euler's Theorem
 and equation \eqref{eq:euler2}.  We add this vector to the first column of $H_h(\x)$, which does not change $D_h(\x)$, thus
\begin{align*}
D_h(\x) &=  \begin{vmatrix}
0 & \nabla h^T(\x)\\
\nabla h(\x) & \nabla^2 h(\x)
\end{vmatrix} 
= 
 \begin{vmatrix}
d\, h(\x) & \nabla h^T(\x)\\
d\, \nabla h(\x)  & \nabla^2 h(\x)
\end{vmatrix} 
=d 
 \begin{vmatrix}
 h(\x) & \nabla h ^T(\x)\\
\nabla h(\x)  & \nabla^2 h(\x)
\end{vmatrix}. 
\end{align*}
Expanding about the left column, we separate this into two determinant computations
\begin{align*}
D_h(\x) &=  d \,
h(\x) \det(\nabla^2 h(\x)) 
+ d
 \begin{vmatrix}
 0 & \nabla h^T(\x)\\
\nabla h(\x)  & \nabla^2 h(\x)
\end{vmatrix}
= 
 d \,
h \det(\nabla^2 h(\x)) 
+ d\,
 D_h(\x). 
\end{align*}
Solving for $D_h(\x)$ finishes the result.
\end{proof}

\begin{proof}[Proof of Corollary~\ref{cor:Df_homogeneous}.] 
Since $f$ is homogeneous translatable, there exists a $\t \in \R^2$ such that $f(\x + \t) = h(\x)$ for some homogeneous polynomial $h$ of degree $d$.
By Lemma~\ref{lemma:Df-det},  
\begin{equation}
D_h(\x) = \frac{-d}{d-1}h(\x) \det(\nabla^2h(\x)) = \frac{-d}{d-1} h(\x) \left(  h_{11}(\x)h_{22}(\x) -  h_{12}^2(\x)\right) .
\end{equation}
By applying Euler's Theorem to the partial derivatives, we find that the second partial derivatives are homogeneous of degree $d-2$.  Since products of homogeneous functions are homogeneous of the degrees added and sums of homogeneous functions are homogeneous provided they have the same degree, we see that $D_h$ is homogeneous of degree $d + (d-2) + (d-2) = 3d-4$.   
If $\det(\nabla^2 h)$ is the zero function, then $D_h$ is actually the zero function.
Therefore, $D_h$ is either a homogeneous polynomial of degree $3d-4$, or it is the zero function.  
Finally, notice that $D_f(\x+ \t) = D_h(\x)$.  Therefore $D_f$ is also homogeneous translatable.
\end{proof}

We last show how to check in polynomial time if a polynomial $f\colon \R^2 \to \R$ is homogeneous translatable.

\begin{proposition}
\label{prop:ht}
Let $f \colon \R^2 \to \R$ be a polynomial of degree $d\geq 1$ given by 
$
f(\x) = \sum_{\v \in \Z^2_+, \|\v\|_1 \leq d} c_\v \x^{\v}
$
where $c_\v \in \Z$.  In polynomial time in the size of the coefficients $c_\v$, we can determine if $f$ is homogeneous translatable and if so, compute a rational translation vector $\t$ such that $f(\x+\t)$ is a homogeneous polynomial. 
\end{proposition}

\begin{proof}
We begin by applying an invertible linear transformation $T$ to the variables such that $f(T\x)$ such that $f(T\x) = \sum_{\v \in \Z^2_+, \|\v\|_1 \leq d} \bar c_\v \x^{\v}$ where $\bar c_\v = 0$ for some $\v \in \Z^2_+$ with $\| \v \|_1 = d$.  If $f$ already has this property, then we take $T = I$.  Otherwise, $c_{(d,0)}, c_{(d-1,1)} \neq 0$.  Then we choose $T$ such that $T\x = (x_1 - c_{(d-1,1)} x_2 / (dc_{d,0}), x_2)$.  With this choice, $\bar c_{(d-1,1)} = 0$.  
Since homogeneity is preserved under linear transformations, $f$ is homogeneous translatable if and only if $f \circ T$ is homogeneous translatable.

Now that $\bar c_\v = 0$ for some $\v\in \Z^2_+$ with $\|\v\|_1 = d$, there must exist a $\bar \v\in \Z^2_+$ with $\|\bar \v\|_1 = d$ and $\bar c_{\bar \v} = 0$ such that either $\bar c_{(\bar v_1 + 1, \bar v_2 - 1)} \neq 0$ or $\bar c_{(\bar v_1 -1, \bar v_2 + 1)} \neq 0$.  Fix such a $\bar \v$ and assume, without  loss of generality, that $\bar c_{(\bar v_1 + 1, \bar v_2 -1)} \neq 0$.

We consider the monomials of degree $d-1$ (in terms of the $\x$ variables) of the expanded version of $f(T\x + \t)$. 
The coefficient on the monomial $\x^\v$ for any $\|\v\|_1 = d-1$ must vanish for $\t$ to be a desired translation; hence we must have the equation
$$
\bar c_\v + { v_1 + 1 \choose 1} \bar c_{(v_1 + 1, v_2)} t_1 + {v_2 + 1 \choose 1} \bar c_{(v_1, v_2 + 1)} t_2 = 0.
$$
In particular, this relation for $ \v = (\bar v_1, \bar v_2 - 1)$ shows that $t_1 = -\bar c_{(\bar v_1, \bar v_2 -1)} \left(  { v_1 + 1 \choose 1} \bar c_{(\bar v_1 + 1, \bar v_2 - 1} \right)^{-1}$.
If there is any relation where the coefficient on $t_2$ is nonzero, then $t_2$ is determined by this equation and  by $t_1$.
Otherwise, if  $t_2$ does not appear in any coefficients of $f(T\x + \t)$, then it suffices to choose $t_2 = 0$ since this value has no effect on $f(T\x + \t)$.   If instead $t_2$ only appears in lower degree monomials, then $f$ is not homogeneous translatable since the coefficient of  $\x^{\hat \v}$ in the expansion of $f(T\x + \t)$, where $\hat \v \in \argmax \{ \|\v\|_1 : \v \in \Z^2_+, \|\v\|_1 \leq d, c_\v \neq 0, v_2 \neq 0\}$, is exactly $c_{\hat \v} \neq 0$ and is not affected by the translation $\t$.

Finally, by computing all coefficients in the expanded version of $f(T\x + \t)$, we can test if $f(T\x+\t)$ is homogeneous hence verify whether or not $f\circ T$ is homogeneous translatable.  If so, then $f$ is homogeneous translatable with translation vector $T^{-1} \t$.  All the operations done in these calculations can be carried out in polynomial time in the size of the coefficients $c_{\v}$.  
\end{proof}

\end{APPENDICES}

\section*{Acknowledgments}
We would like to thank Amitabh Basu for the discussions about Theorem~\ref{thm:dd-min}.

\bibliographystyle{amsplain}
\bibliography{references}

\end{document}